\documentclass[11pt]{amsart}

\usepackage{times}
\usepackage{amsmath,  amssymb,slashed,url,bm,upgreek,amsthm}
\usepackage{graphicx,enumerate}
\usepackage{hyperref}
\newcounter{intro}

\newtheorem{theo}[intro]{Theorem}

\newtheorem{propo}[intro]{Proposition}

\newtheorem{thm}{Theorem}[section]
\newtheorem{lem}[thm]{Lemma}
\newtheorem{prop}[thm]{Proposition}
\newtheorem{cor}[thm]{Corollary}
\newtheorem{defi}[thm]{Definition}
\newtheorem{rem}[thm]{Remark}
\newtheorem{rems}[thm]{Remarks}

\newtheorem*{merci}{Acknowledgements}

\newcommand{\cref}[1]{Corollary~\ref{#1}}

\newcommand{\lref}[1]{Lemma~\ref{#1}}
\newcommand{\pref}[1]{Proposition~\ref{#1}}
\newcommand{\rref}[1]{Remark~\ref{#1}}
\newcommand{\tref}[1]{Theorem~\ref{#1}}

\DeclareMathOperator{\Ka}{K}\DeclareMathOperator{\ka}{k}

\DeclareMathOperator{\Qc}{Q}

\DeclareMathOperator{\supp}{supp}

\DeclareMathOperator{\ricci}{Ricci}
\DeclareMathOperator{\ricm}{Ric_-}
\DeclareMathOperator{\scal}{Scal}
\DeclareMathOperator{\diam}{diam}

\DeclareMathOperator{\dv}{dv}
\DeclareMathOperator{\dm}{dm}\DeclareMathOperator{\m}{m}
\DeclareMathOperator{\vol}{vol}
\def\R{\mathbb R}\def\N{\mathbb N}\def\bS{\mathbb S}\def\bC{\mathbb C}

\def\cC{\mathcal C}
\def\cD{\mathcal D}

\def\cH{\mathcal H}

\def\loc{\mathrm{loc}}
\def\sch{Schr\"odinger }
\def\cO{\mathcal O}

\def\cV{\mathcal V}
\def\cU{\mathcal U}
\def\ra{\rangle}
\def\la{\langle}
\begin{document}
\title[Geometric inequalities and Ricci curvature in the Kato class]{Geometric inequalities for manifolds with Ricci curvature in the Kato class}
\author{ Gilles Carron }
\address{Laboratoire de Math\'ematiques Jean Leray (UMR 6629), Universit\'e de Nantes, CNRS, \'Ecole Centrale de Nantes
2, rue de la Houssini\`ere, B.P.~92208, 44322 Nantes Cedex~3, France}
\email{Gilles.Carron@univ-nantes.fr}
\begin{abstract}
We obtain an Euclidean volume growth results for complete Riemannian manifolds satisfying a Euclidean Sobolev inequality and a spectral type condition on the Ricci curvature. We also obtain eigenvalue estimates, heat kernel estimates, Betti number estimates for closed manifolds whose Ricci curvature is controlled in the Kato class.
\par
R\'ESUM\'E: 
On d\'emontre qu'une vari\'et\'e riemannienne compl\`ete v\'erifiant une in\'egalit\'e de Sobolev euclidienne et dont la courbure de Ricci est petite dans une classe de Kato est \`a croissance euclidienne du volume. On obtient aussi des estimations spectrales, du noyau de la chaleur et du premier nombre de Betti des vari\'et\'es riemanniennes compactes dont la courbure de Ricci est control\'ee dans une classe de Kato.
\end{abstract}
\date\today
\maketitle

\footnotetext{\emph{Mots cl\'es: } In\'egalit\'e de Sobolev, croissance du volume, noyau de Green, transform\'ee de Doob}

\section{Introduction}
One of the first motivation was a quest for a higher dimensional analogue of a beautiful result of P. Castillon (\cite{Cast}):
\begin{thm}\label{castillon}Let $(M^2,g)$ be a complete Riemannian surface such that for some $\lambda>\frac14$, the \sch operator
$$\Delta+\lambda K_{g}$$ is non negative:
$$\forall \psi\in \cC^{\infty}_0(M)\colon \,\lambda\int_{M}K_{g}\psi^{2}dA_{g}\le \int_{M}|d\psi|_{g}^{2}dA_{g}$$
then for all $R\ge 0$ and all $x\in M$:
$$\mathrm{Aera}_{g}\left(B(x,R)\right)\le c(\lambda)\, R^{2}.$$
\end{thm}
In fact, P. Castillon has shown that such a surface is either conformally equivalent to the plane or the cylinder:
$$(M^2,g)\simeq_{\mathrm{conf.}} \bC\ \mathrm{or}\ (M^2,g)\simeq_{\mathrm{conf.}}  \bC^*;$$
and also he shows that a similar conclusion holds when (for some $\lambda>1/4$) the operator $\Delta+\lambda K_{g}$ has a finite number of negative eigenvalues, but  in that case
$\sup_R\frac{\mathrm{Aera}_{g}\left(B(x,R)\right)}{R^2}$ could be unbounded on $M$.

But what caught our particular attention is the upper Euclidean bound on the volume of geodesic balls; it is sometimes crucial to get such a bound; for instance this kind of estimate was one of the difficult result obtained by G. Tian and J. Viaclovsky (\cite{TV1}) and this was a key point toward the description of the moduli spaces of critical Riemannian metrics in dimension four (\cite{TV2}).

One knows some analytical criteria that guarantees a lower Euclidean bounds. For instance if $(M^n,g)$ is a complete Riemannian manifold satisfying the Euclidean Sobolev inequality (\cite{aku,carronlambda})
\begin{equation}\label{SEu}\tag{Sob}
\forall \psi\in \cC_0^\infty(M)\colon\ \mu \left(\int_M \psi^{\frac{2n}{n-2}}\dv_g\right)^{1-\frac 2n}\le \int_M |d\psi|_g^2\dv_g\end{equation}
then one has for all $x\in M$ and $R>0$:
$$c_n \mu^{\frac n2} R^n\le \vol B(x,R).$$
The same conclusion holds in presence of log-Sobolev inequality (\cite[Proposition 5.1]{Ni1}).  But I am not aware of other geometrical-analytical conditions than the Castillon's results in dimension 2 or of the Bishop-Gromov comparison theorem insuring the Euclidean volume growth. Not only, the  Bishop-Gromov comparison theorem implies that  a complete Riemannian manifold $(M^n,g)$ with non negative Ricci curvature:
$$\ricci\ge 0$$ has at most Euclidean growth:
$$\forall x\in M,\,R>0:\ 
 \vol B(x,R)\le \omega_n  R^n,$$
 But this hypothesis also implies stronger properties:  \begin{enumerate}[-]
 \item $(M^n,g)$
 is doubling: there is a constant $\uptheta$ such that 
 \begin{equation}\label{ddoubling}\tag{D}\forall x\in M,\,R>0:\ 
 \vol B(x,2R)\le\uptheta \vol B(x,R)
 \end{equation} 
  \item $(M^n,g)$
 satisfies the Poincar\'e inequalities (\cite{buser}): there is a constant $\upgamma$ such that for any geodesic ball $B$ of radius $r$, we have\footnote{where we have noted $\psi_B=\frac{1}{\vol B} \int_B \psi \dv_g.$}
\begin{equation}\label{Poin}\tag{SP}
\forall \psi\in \cC^1(B)\colon\ \int_B (\psi-\psi_B)^2\dv_g\le \upgamma r^2 \int_B |d\psi|_g^2\dv_g\end{equation}

 \item The heat kernel $\{H(t,x,y)\}_{(t,x,y)\in \R_+\times M\times M}$ of $(M^n,g)$ satisfies the Li-Yau estimate \cite{LY}:
 \begin{equation}\label{LY}\tag{LY}
 \frac{c}{\vol B(x,\sqrt{t})}e^{-\frac{d(x,y)^2}{c t}}\le H(t,x,y)\le  \frac{C}{\vol B(x,\sqrt{t})}e^{-\frac{d(x,y)^2}{C t}}
 \end{equation}
 \item And according to D. Bakry (\cite{Bakry}), for every $p\in (1,+\infty)$, the Riesz transform $d\Delta^{-\frac12}\colon L^2(M)\rightarrow L^2(T^*M)$ has a bounded $(L^p\to L^p)$ extension.
 \end{enumerate}
 
 In fact according to (\cite{G91,Saloff-Coste}), the estimates ( \ref{LY} ) are equivalent to the conditions (\ref{ddoubling} and \ref{Poin}). Hence the non negativity of the Ricci curvature or more generally a lower bound on the Ricci curvature provides a lot a good properties;  one also gets now sophisticate results about  the singular sets of Gromov-Hausdorff limits of Riemannian manifolds under Ricci curvature lower bound (\cite{CC,CCT}) and the lack of such tools is a major obstacle to a generalization of the result of G. Tian and J. Viaclosky in higher dimension.
 
 We are looking for spectral or analytical conditions that would guarantee at most Euclidean growth and others properties that  are  markers of the non negative Ricci curvature condition. We are studying the familly of \sch operators
 $$L_\lambda=\Delta-\lambda \ricm$$ where
 $\ricm$ is  the function on $M$ defined by
$$-\ricm(x):=\max\{\kappa(x), 0\}$$ where 
$$\kappa(x):=\inf_{\vec v\in T_{x}M, g_x(\vec v,\vec v)=1} \ricci_{x}(\vec v,\vec v).$$

I do not think that it is possible to show a result of the type of \tref{castillon} in higher dimension under a sole assumption of a non negativity in $L^2$ of $L_\lambda$:
$$\forall \psi\in \cC_0^\infty(M)\colon\ \lambda\int_M \ricm \psi^2\dv_g\le \int_M |d\psi|^2\dv_g.$$
On a non compact Riemannian manifold, the behavior of the heat semi group associated to a \sch operator $\Delta-V$ can be quite different on $L^2$ and on $L^p$. For instance the fact that the heat semi group is uniformly bounded on $L^\infty$:
$$\sup_{t>0} \left\| e^{-t(\Delta-V)}\right\|_{L^\infty\to L^\infty}<+\infty$$ implies the non negativity on $L^2$ of the \sch operator $\Delta-V$, but it can happen that the  \sch operator $\Delta-V$ is non negative on $L^2$ but the semi group is not uniformly bounded on $L^\infty$ (\cite{DS}).

When $(M,g)$ is stochastically complete, for instance (see \cite{GriBAMS}) when one has for some
$o\in M$:
$$\vol B(o,R)\le c e^{c R^2},$$
then the  uniformly boundedness on $L^\infty$ of the semi group $\left(e^{-tL_\lambda}\right)_{t>0}$ is equivalent to the {\it gaugeability} of the \sch operator $L_\lambda$:
there is a function $h\colon M\rightarrow \R$ such that $1\le h\le \gamma$ and
$$\Delta h-\lambda \ricm h=0.$$
And our first result is 
\begin{theo}\label{global} If $(M,g)$ is a complete Riemannian manifold that satisfies the Euclidean Sobolev inequality (\ref{SEu})  and such that for some $\delta>0$ the \sch operator $\Delta-(n-2)(1+\delta)\ricm$ is gaugeable then there is a constant $\uptheta$ that depends only on the dimension, the Sobolev constant $\mu$, $\delta$ and the gaugeability constant $\gamma$ such that for all $x\in M$ and $R\ge 0$:
$$ \frac{1}{\uptheta}\, R^{n}\le \vol B(x,R)\le \uptheta\, R^{n}.$$
\end{theo}
We already said that the Euclidean Sobolev inequality yields the Euclidean lower bound, the crucial point here is to get the upper bound.  According to a celebrated result of N. Varopoulos (\cite{varo}), the Euclidean Sobolev inequality is equivalent 
to an Euclidean type upper bound on the heat kernel:
$$\forall t>0,x,y\in M\colon\ H(t,x,y)\le  \frac{C}{t^{\frac n 2}}e^{-\frac{d(x,y)^2}{C t}}.$$
Moreover a nice observation by T. Coulhon (\cite{C_JLMS}) is that the lower bound
$$\forall t>0,x,y\in M\colon\ H(t,x,y)\ge  \frac{c}{t^{\frac n 2}}e^{-\frac{d(x,y)^2}{c t}}.$$ yields an Euclidean upper bound on the volume of geodesic balls. 

 In fact,  Riemannian manifolds considered in the above theorem satisfy many properties of manifolds with non negative Ricci curvature:
\begin{theo}\label{globalsuite} If $(M,g)$ is a complete Riemannian manifold that satisfies the condition of the \tref{global}, then:
\begin{enumerate}[-]
\item $(M,g)$ is doubling and satisfies the Poincar\'e inequalities
\item The heat kernel of $(M,g)$ satsifies the Li and Yau estimates (\ref{LY}).
\item For any $p\in (1,+\infty)$,  the Riesz transform $d\Delta^{-\frac12}\colon L^p(M)\rightarrow L^p(T^*M)$ is a bounded operator.
\end{enumerate}
\end{theo}

The gaugeability condition made in the \tref{global} is insured when the Ricci curvature is small at infinity (short range potential) and when the \sch operator $\Delta-(n-2)(1+\delta)\ricm$ is non negative in $L^2$. More precisely, any of the  following conditions imply this gaugeability condition:
\begin{enumerate}[a)]
\item From \cite[Example 2.2]{DevKato} : There is some $\upepsilon\in (0,1)$ such that $\ricm\in L^{\frac{n}{2}(1-\upepsilon)}$ and for $\mu$ is the Sobolev constant in (\ref{SEu}): $$\int_M \ricm^{\frac n2} \dv_g\le \mu^{\frac{n}{2}} (1-\upepsilon).$$
\item   There is some $\upepsilon\in (0,1)$ such that $\ricm\in L^{\frac{n}{2}(1-\upepsilon)}$ and the \sch operator $\Delta-(n-2)(1+\delta)\ricm$ is non negative in $L^2$ .
\item The function $\ricm$ belongs to the Kato class and its Kato constant is smaller than $1/(n-2)$; that is to say when$G(\bullet,\bullet)$ is the positive minimal Green kernel of $(M^n,g)$:
$$\Ka(\ricm):=\sup_{x\in M} \int_M G(x,y)\ricm(y)\dv_g(y)\, <\frac{1}{n-2}.$$
\item For some constant $\upepsilon_n$ depending only on $n$:
$$\sup_{x\in M} \int_0^\infty\frac{1}{r^{n-1}}\left(\int_{B(x,r)} \ricm(y)\dv_g(y)\right)dr\, <\upepsilon_n.$$
\end{enumerate}

We have obtain similar results under the condition that the \sch operator $\Delta-(n-2)(1+\delta)\ricm$ is jaugeable outside a compact set:
\begin{theo}\label{globalbis}If $(M,g)$ is a complete Riemannian manifold that satisfies the Euclidean Sobolev inequality.\par
Assume for some $\delta>0$ the \sch operator $\Delta-(n-2)(1+\delta)\ricm$ is gaugeable at infinity: there is a compact subset $K\subset M$ and a bounded positive function $h\colon M\setminus K\rightarrow \R_+$ such that
$$\Delta h-(n-2)(1+\delta)\ricm h=0\ \mathrm{and}\ 1\le h\le \gamma.$$

Then for each $o\in M$,  there is a constant $\uptheta=\uptheta_o$ such that for all  $R\ge 0$:
$$ \frac{1}{\uptheta}\, R^{n}\le \vol B(o,R)\le \uptheta\, R^{n}.$$\par 
Moreover if the Kato constant of $\ricm$ is small at infinity :
$$\sup_{x\in M\setminus K} \int_{M\setminus K} G(x,y)\ricm(y)\dv_g(y)\, <\frac{1}{16n},$$ then 
\begin{enumerate}[-]
\item $(M,g)$ is doubling,
\item the heat kernel of $(M,g)$ satisfies the upper Li and Yau estimates (\ref{LY}):
$$\forall x,y\in M, t>0\colon\ H(t,x,y)\le  \frac{C}{\vol B(x,\sqrt{t})}e^{-\frac{d(x,y)^2}{C t}},$$
\item When $n\ge 4$ and $p\in (1,n)$,  the Riesz transform $d\Delta^{-\frac12}\colon L^p(M)\rightarrow L^p(T^*M)$ is a bounded operator.
\end{enumerate}
\end{theo}
According to \cite{DevKato}, the assumptions of the above \tref{globalbis} are satisfied by  complete Riemannian manifolds satisfying an Euclidean Sobolev inequality and such  that for some $\upepsilon\in (0,1)$ such that $\ricm\in L^{\frac{n}{2}(1-\upepsilon)}$.
Our ideas can also be adapted to understand the geometry of geodesic ball where we get some control of the Ricci curvature in some stronger Kato class:

\begin{theo}\label{local}Let $(M^n,g)$ be a Riemannian manifold and assume that $B(o,3R)\subset M$ is a relatively compact geodesic ball. Let $p>1$ and let $q:=p/(p-1)$. Assume  that for some $\delta>\frac{\left(q(n-2)-2\right)^{2}}{8q(n-2)}$, the operator $\Delta-(1+\delta)(n-2)\ricm$ is non negative:
$$\forall \psi\in \cC_0^{\infty}(B(o,3R))\colon\ (1+\delta)(n-2)\int_{B(o,3R)}\ricm\psi^{2}\dv_{g}\le \int_{B(o,3R)}|d\psi|^{2}\dv_{g}.$$
And introduce the Sobolev constant $\mu$:
$\forall \psi\in \cC_0^{\infty}(B(o,3R))$:
$$\mu \|\psi\|^{2}_{\frac{2n}{n-2}}\le \|d\psi\|_{2}^{2} .$$ 
Let $G$  the Green kernel for the Laplacian $\Delta$ with the Dirichlet boundary condition of $B(o,3R)$, we introduce the $L^p$ Kato constant for $\ricm$:
$$\Ka_p(\ricm):=R^{2p-2}\sup_{x\in B(o,3R)} \int_{B(o,3R)}G(x,y)\ricm(y)^p\dv_{g}(y).$$
Then there are constants $\uptheta$ and $\upgamma$ that depends only on $n,p,\delta$ the Sobolev constant $\mu$, of the Kato constant $\Ka_p(\ricm)$ and of the volume density $\frac{\vol B(o,3R)}{R^n}$ such that for any $x\in B(o,R)$ and any $r\in (0,R)$:
$$\frac 1 \uptheta r^n\le \vol B(o,r)\le \uptheta r^n,$$
moreover, all these balls $B=B(x,r)$ satisfy the Poincar\'e inequality :$$\forall \psi\in \cC^1(B)\colon\ \int_B (\psi-\psi_B)^2\dv_g\le \upgamma r^2 \int_B |d\psi|_g^2\dv_g$$
\end{theo}

Recents papers have emphasized how a control on the Ricci curvature in a Kato class can be useful in order to control some geometrical quantities for closed or complete Riemannian manifolds (\cite{CZ,CDS,DevKato,Rose1, Rose2, RoseStollmann,ZZhu1,ZZhu2}). For a closed Riemannian manifold $(M,g)$, the Green kernel is not positive and one needs alternative definitions for the Kato class, one solution is to consider the Green kernel of $\Delta+\frac{1}{R^2}$ and another one is to consider the parabolic Kato class defined with the help of the hear kernel. Both approachs are equivalents (see \tref{equiKato}). We have noticed that the works of Qi S. Zhang and M. Zhu \cite{ZZhu1} together with some classical ideas can be used in order to obtain geometric and topological estimates based on a Kato bound for the Ricci curvature. Recently C. Rose has also obtained similar results based on this idea (\cite{Rose2}. 
When $(M,g)$ is a closed Riemannian manifold of diameter $D$, we define a scaled invariant geometric quantity $\xi(M,g)$ to be the smallest positive real number such that for all $x\in M$:
$$\int_0^{\frac{D^2}{\xi^2}}\int_M H(t,x,y)\ricm(y)\dv_g(y)dt\le \frac{1}{16n}.$$
For instance if the Ricci curvature is bounded from below:
$$\ricci\ge -(n-1)\kappa^2g$$ then
$$\xi(M,g)\le \kappa D.$$

\begin{theo}\label{LiYau}There is a constant $\upgamma_n$ that depends only on $n$ such that if $(M,g)$ is a closed Riemannian manifold of dimension $n$ and diameter $D$ then
\begin{enumerate}[i)]
\item The first non zero eigenvalue of the Laplacian $\lambda_1$ satisfies
$$\lambda_1\ge \frac{\gamma_n^{-1-\xi(M,g)}}{D^2}.$$
\item The first Betti number of $M$ satisfies
$$b_1(M)\le n+\frac{1}{4}+\xi(M,g)\gamma_n^{1+\xi(M,g)}.$$
In particular, there is a $\upepsilon_n>0$ such that if $\xi(M,g)<\upepsilon_n$ then
$$b_1(M)\le n.$$
\item  $(M,g)$ is doubling: for any $x\in M$ and $0\le R\le D/2$:
$$\vol B(x,2R)\le \gamma_n^{1+\xi(M,g)} \vol B(x,R).$$
\item For all $t>0$ and $x\in M$:
$$H(t,x,x)\le \frac{\gamma_n^{1+\xi(M,g)}}{\vol B(x,\sqrt{t})}. $$
\end{enumerate}
\end{theo}

We can also introduce the $L^p$ Kato constant of $\ricm$:
$$\ka_{p,T}(\ricm)^{p}=(\diam M)^{2p-2}\sup_{x\in M}\int_{0}^{T}H(s,x,y)\ricm^{p}(y)\dv_{g}ds.$$
When $p>1$, we get a slight improvement of the previous theorem:
\begin{propo}if $(M,g)$ is a closed Riemannian manifold of dimension $n$ and diameter $D$ and $p>1$ then (with $q=p/(p-1)$:
$$\xi(M,g)\le \upalpha(M,g,T,p):=\max\left\{ \frac{D}{\sqrt{T}}, \left(16n \ka_{p,T}(\ricm)\,\right)^{q/2}\right\}.$$
Moreover there is a constant $\uptheta$ depending only on $\upalpha(M,g,T,p)$ and $n$ such that
 for any $x\in M$ and $0\le r\le  R\le D$:
$$\frac{\vol B(x,R)}{R^n}\le \uptheta \frac{\vol B(x,r)}{r^n}\le \uptheta^2.$$
\end{propo}

A quick comparison between the results obtained in the case of closed manifold and complete's one yields naturally the question of the utility the Euclidean Sobolev inequality in \tref{global}. In fact a remark made by Qi S. Zhang and M. Zhu in  \cite{ZZhu1} implies that the results obtained in the \tref{LiYau} could be generalized on complete Riemannian manifold provided one gets good approximations of the distance function : for some $c>0$ and for all $o\in M$, there is $\chi_o\colon M\rightarrow \R_+$ such that 
$$d(o,x)/c\le \chi_o(x)\le c d(o,x)$$
$$|d\chi_o|^2+\chi_o\left( \Delta \chi_o\right)\le c.$$
But this is very strong hypothesis. Our hypothesis made on the Sobolev inequality make possible the comparison between the level set of the Green kernel and of the geodesic spheres.

Our estimates on the first Betti number is a generalization of the one obtained by  M. Gromov under a lower bound on the Ricci curvature. According to T. Colding \cite{Colding1, CC}, one knows that there exists an $\upepsilon(n)>0$ such that if $(M^n,g)$ is a closed $n$-dimensional manifold with $\ricm \diam^2(M)>\upepsilon(n)$  and $b_1(M)=n$, then $M$ is diffeomorphic to a torus $\mathbb{T}^n$. Hence it is quite natural to ask what could be said about closed Riemannian manifold satisfying $\xi(M,g)<<1$ and $b_1(M)=n$, we believe that such a manifold should be closed in the Gromov-Hausdorff topology to a torus $\mathbb{T}^n$. But in order to said more, it should be useful to understand the space that are Gromov-Hausdorff limit of Riemannian manifolds $(M^n,g)$ with $\xi(M,g)\le \Xi$ and $\diam M\le D$. 
Note  that our results yields a pre compactness result in the Gromov-Hausdorff topology for these class of spaces.

A lower bound on the Ricci curvature also yields some isoperimetric inequalities and an interesting question is to know wether a control of the Ricci curvature in some Kato class yields some isoperimetric inequality.  

In  pioneering paper  (\cite{GallotInt}),  S. Gallot has proven isoperimetric inequalities, eigenvalues and heat kernel estimates for closed  Riemannian manifold $(M^n,g)$ under a control of $\ricm$ in $L^p$ (for $p>n/2$). It should be interesting to know wether one can get a control of the Ricci curvature in some Kato class from a control of $\ricm$ in $L^p$ (for $p>n/2$).

In the next section, we will review and collect some analytical tools that will be used in the paper, for instance we will describe some Agmon's type volume estimate mainly due to P. Li and J.Wang (\cite{LW1,LW2}) that will be crucial for the proof of the \tref{global}. We will also proved a new elliptic estimate based on a variation of the De Giorgi-Nash-Moser iteration scheme. The third section is devoted to the proof of  the \tref{LiYau}.  The \tref{global} and the first part of \tref{globalbis} will be proved in the fourth section and the \tref{local} will be proved in the last section.  
\begin{merci} I wish to thank F. Bernicot, P. Castillon, B. Devyver and  M. Herzlich   for valuable conversations. I was partially supported by the ANR grant: {\bf ANR-12-BS01-0004}: {\em Geometry and Topology of Open manifolds}.
\end{merci}
\tableofcontents\section{Preliminary}
In this section we review some analytical objects, tools and results that will be used throughout the paper.
We consider $(M,g)$ a Riemannian manifold and $\Phi$ a positive Lipschitz function on $M$ and the measure $\dm=\Phi \dv_{g}$. Le $L^{p}$ norm associated to this measure will be noted $\|\bullet \|_{p}$ or $\|\bullet \|_{\m,p}$.
\subsection{Laplacian}
\subsubsection{}The Laplacian $\Delta_{\m}$ or $\Delta_{\Phi}$ is the differential operator associated to the quadratic form:
\begin{equation}\label{QFs}\tag{QF} \psi\in \cC^{\infty}_{0}(M)\longmapsto  q(\psi):=\int_{M} |d\psi|_{g}^{2}\dm,
\end{equation}
by the Green formula: $$\forall  \psi\in \cC^{\infty}_{0}(M)\colon \int_{M} |d\psi|_{g}^{2}\dm=\int_{M} \left(\Delta_{\m}\psi\right)\psi \dm.$$
The geometric Laplacian will be noted $\Delta=\Delta_{1}$ and we have the formula
$$\Delta_{\m}\psi=\Delta \psi-\la d\log \Phi,d\psi\ra_{g}.$$

The Friedrichs realization of operator $\Delta_{\m}$ is associated to the minimal extension of the above quadratic form i.e. if $\cD(q)$ is the completion of $\cC^{\infty}_{0}(M)$ for the norm $\psi\mapsto\sqrt{q(\psi)+\|\psi\|_{2}^{2}}$, then 
$$\cD\left(\Delta_{\m}\right)=\left\{v\in \cD(q),\ \exists C\ \mathrm{such\ that}\ \forall\varphi\in \cC^{\infty}_{0}(M)\colon |\la v,\Delta_{\m}\varphi\ra\ \le C\|\varphi\|_2\right\}\ .$$

\begin{rems}\begin{enumerate}[-]
\item If $(M,g)$ is geodesically complete, then $$\Delta_{\m}\colon \cC^{\infty}_{0}(M)\longrightarrow L^{2}(M,\dm)$$
has an unique selfadjoint extension. 
\item If $M$ is the interior of a compact manifold with boundary $M=X\setminus \partial X$ and if $g$ and $\Phi$ have Lipschitz extension to $X$ then the Friedrichs realization of operator $\Delta_{\m}$ is the Laplacian associated to the Dirichlet boundary condition.
\end{enumerate}
\end{rems}
\subsubsection{Chain rule}\label{chainrule}When $v\in \cC^{\infty}(M)$ and $f\in \cC(\R,\R)$, by a direct computation, we have
$$\Delta_{\m} f(v)=f'(v)\, \Delta_{\m}v-f''(v)|dv|_{g}^{2}.$$In particular if $f$ is non decreasing, convex and if 
$\Delta_{\m}v\le V$ where $V$ is a non negative function then $\Delta_{\m} f(v)\le f'(v)\,V$.
By approximation this can be generalized to weak solution and non smooth convex function.
\begin{lem}\label{ChainR} Let $v\in W_{\loc}^{1,2}$ and $V\in L^{1}_{\loc}$ a non negative function satisfying 
$$\Delta_{\m}v\le V\ \mathrm{weakly},$$
then for every non decreasing, convex function $f$, we have 
$$\Delta_{\m} f(v)\le f'(v)\,V\ \mathrm{weakly}.$$
\end{lem}
Recall that we say that 
$$\Delta_{\m}v\le V\ \mathrm{weakly}$$ if for any non negative $\varphi\in C^{\infty}_{0}(M)$:
$$\int_{M}v\Delta_{\m}\varphi \dm\le \int_{M}V\varphi \dm,$$
or equivalently
if for any non negative
$\varphi\in C^{\infty}_{0}(M)$ (or $\varphi\in \cD(q)$):
$$\int_{M}\la dv, d \varphi\ra_{g}\dm\le \int_{M}V\varphi \dm.$$
\proof
When $v\in W^{1,2}_{\loc}$ and $f$ is smooth, then $f(v), f'(v)\in W^{1,2}_{\loc}$, and if 
$ \varphi\in C^{\infty}_{0}(M)$ then $f'(v)\varphi\in \cD(q)$. Hence
\begin{equation*}
\begin{split}
\int_{M}f(v)\Delta_{\m}\varphi \dm&=\int_{M}\la f'(v)dv,d\varphi \ra_{g}\dm\\
&\int_{M}\la dv, d \left(f'(v)\varphi\right) \ra_{g}\dm-\int_{M}f''(v)|dv|_{g}^{2}\varphi\dm.\end{split}\end{equation*}
Now if moreover $f$ satisfies $f', f''\ge 0$, then we get
for any $\varphi\in C^{\infty}_{0}(M)$ such that $\varphi\ge 0$:
$$\int_{M}f(v)\Delta_{\m}\varphi \dm\le \int_{M}V\varphi \dm.$$
When $f$ is not smooth but non decreasing, convex then we can approximate $f$ by a sequence of smooth non decreasing, convex function. Indeed if $\rho$ is a smooth non negative function with support in $(0,1)$ with integral $1$, then $$f_{\ell}(x)=\int_{\R}f\left(x+\frac1\ell t\right)\rho(t)dt,$$ defined a sequence of smooth, non decreasing, convex functions converging to $f$.
\endproof
\noindent{\bf Examples:}
\begin{enumerate}[-]	
\item For $\alpha \ge 1$, let $x\mapsto x_{+}^{\alpha}=\max\{x,0\}^\alpha$. If $v\in W^{1,2}_{\loc}$ and $V\in L^{1}_{\loc}$  are such that $V\ge 0$ and 
$\Delta_{\m}v\le V$ then $$\Delta_{\m}v_{+}^{\alpha}\le \alpha Vv_{+}^{\alpha-1}. $$
\item If $\alpha \ge 1$ and $x\mapsto g_{\alpha}(x)= (x-1)_{+}^{\alpha}=\max\{x-1,0\}^\alpha$, then if $v\in W^{1,2}_{\loc}$ and $V\in L^{1}_{\loc}$are such that $V\ge 0$ and 
$\Delta_{\m}v\le V$ then $$\Delta_{\m} g_{\alpha}(v)\le \alpha V(v-1)_{+}^{\alpha-1}. $$
\end{enumerate}

\subsubsection{Integration by part formula}\label{IPP}
The formula
$$|d(\chi v)|^{2}_{g}=|d\chi|^{2}_{g}v^{2}+\la dv, d\left(\chi^{2}v\right)\ra_{g},$$ implies the  following integration by part inequality
\begin{lem}Let $v\in W_{\loc}^{1,2}$ and let $V\in L^{1}_{\loc}$ be a non negative function such that: 
$$\Delta_{\m}v\le V\ \mathrm{weakly}$$
then for every Lipschitz function $\chi$ with compact support
$$\int_{M}|d(\chi v)|^{2}_{g}\dm\le \int_{M}|d\chi|^{2}_{g}v^{2}\dm+\int_{M}\chi^{2}vV\dm.$$
\end{lem}

In some circonstance, we would like to insure that this inequality is still valid for Lipschitz function that are constant at infinity. The precise notion is the one of parabolicity:
\begin{defi}\label{parabolic}A Borel measure $d\mu$ on a Riemannian manifold is called parabolic if there is  a sequence of smooth functions with compact support $(\chi_k)$ such that:
\begin{enumerate}[-]
\item $0\le \chi_k\le 1$
\item $\chi_k\to 1$ uniformly on  compact sets;
 \item $\lim_{k\to\infty}\int_{M}|d\chi_k|^{2}_{g}d\mu=0$\end{enumerate}
\end{defi}
In this case, we have the 
refinement \begin{lem}\label{lemIIP}Let $v\in W_{\loc}^{1,2}$ and let $V\in L^{1}_{\loc}$ be a non negative function  such that 
$$\Delta_{\m}v\le V\ \mathrm{weakly}.$$
If the measure $v^2\!\dm$ is parabolic, then for every Lipschitz function $\chi$ that is constant outside a compact set 
$$\int_{M}|d(\chi v)|^{2}_{g}\dm\le \int_{M}|d\chi|^{2}_{g}v^{2}\dm+\int_{M}\chi^{2}vV\dm.$$
\end{lem}
\begin{rem}If $(M,g)$ is geodesically complete and if 
$M(r):=\int_{B(o,r)}v^2\dm$ satisfies
$$M(r)=\cO\left(r^{2}\right)\ \mathrm{or} \int^{\infty}_{1}\frac{rdr}{M(r)}=+\infty$$ then the measure $v^2\dm$ is parabolic.
\end{rem}

\subsection{Sobolev inequality}
We recall here some classical results  that hold in presence of the Sobolev inequality.
\begin{thm}\label{sobolev}Let  $(M,g,\m)$ be a weighted Riemanian manifold and assume it satisfies the 
Sobolev inequality $$\forall \psi\in \cC^\infty_{0}\left(M\right)\colon\ \mu \|\psi\|^{2}_{\frac{2n}{n-2}}\le \|d\psi\|^{2}_{2}.$$
Then the following properties hold:
\begin{enumerate}[i)]
\item The heat kernel associated to the Laplacian $\Delta_{\m}$ satisfies
$$\forall x\in M,\forall t>0\colon \, H_{\m}(t,x,x)\le \frac{c_n}{\left(\mu t\right)^{\frac n 2}}.$$
\item The associated positive minimal Green kernel satisfies:\begin{enumerate}[a)]
\item$\displaystyle \forall x,y\in M\colon \, G_{\m}(x,y)\le \frac{c_n}{\mu ^{\frac n 2}}\frac{1}{ d^{n-2}(x,y)}.$
\item$\displaystyle\forall x\in M,\forall t>0\colon\, \m\left(\left\{y\in M; G_{\m}(x,y)>t\right\}\right)\le \left(\mu t\right)^{-\frac{n}{n-2}}.$
\item Hence if $\alpha\in \left(0,n/(n-2)\right)$ and if $\Omega\subset M$ has finite $\m$-measure then
$$\left(\int_{\Omega}G^{\alpha}_{\m}(x,y)\dm(y)\right)^{\frac1\alpha}\le \left(\frac{n}{(n-2)\alpha-n}\right)^{\frac1\alpha}\frac{1}{\mu}\left(\m(\Omega)\right)^{\frac{1}{\alpha}-1+\frac2n}.$$\end{enumerate}
\item If $B(x,r)\subset M$ is a relatively compact geodesic balls in $M$ and if $v\in W^{1,2}_{\loc}(B(x,r))$ satisfies 
$$\Delta_{\m}v\le 0$$ then for $p\ge 2$:
$$|v(x)|^{p}\le \frac{c_{n,p})}{\left(\sqrt{\mu}\, r\right)^n}\int_{B(x,r)}|v|^p(y)\dm(y).$$

\item If $B(x,r)\subset M$ is a relatively compact geodesic ball in $M$  then 
$$c(n)\mu^{\frac n 2}r^{n}\le \m\left(B(x,r)\right).$$

\end{enumerate}
\end{thm}
\begin{rems}

\begin{enumerate}[--]
\item The upper bound on the heat kernel comes essentially from an adaptation in this setting of old ideas of J. Nash (\cite{Nash}). It happens that in fact both properties are equivalent (\cite{varo})
\item This estimate on the heat kernel implies in fact a Gaussian upper bound for the heat kernel:
$$\forall x,y\in M,\forall t>0\colon \, H_{\m}(t,x,y)\le \frac{c_n}{\left(\mu t\right)^{\frac n 2}}e^{-\frac{d^2(x,y)}{5t}},$$ and the formula
$$G_{\m}(x,y)=\int_{0}^{+\infty}H_{\m}(t,x,y)dt$$ yields the estimate ii-a) on the Green kernel.
\item The property ii-b) is in fact equivalent to the Sobolev inequality (\cite{carronlambda}).
\item The elliptic estimate is proved by a classical De Giorgi-Nash-Moser iteration method. And the volume lower bound iv) comes from an application of this estimate on constant function (see \cite{aku,carronlambda})
\end{enumerate}
\end{rems}

\subsection{\sch Operator and the Doob transform}
\subsubsection{\sch Operator}
When $V\in L^{\infty}_{\loc}$ is non negative function such that the quadratic form
$$\psi\in \cC^{\infty}_{0}(M)\longmapsto  q_{V}(\psi):=\int_{M} \left[|d\psi|_{g}^{2}-V\psi^{2}\right]\\dm,
$$ is bounded from below; i.e. there is a constant $\Lambda$ such that 
$$\forall \psi\in \cC^{\infty}_{0}(M): q_{V}(\psi)\ge -\Lambda \int_{M}\psi^{2}\\dm.$$
Then with the Friedrichs extension, we get  a self-adjoint operator that will also be noted:
$$L:=\Delta_{\m}-V.$$
An easy consequence of the maximum principle or of its weak formulation is that if we note $H_{L}$ the heat kernel of the operator $L$ then
$$\forall x,y\in M, \forall t>0\colon \, H_m(t,x,y)\le H_L(t,x,y) .$$
As a consequence if $L$ is subcritical, i.e. if $L$ has a positive minimal Green kernel $G_{L}$ then
$$\forall x,y\in M\colon \, G_m(x,y)\le G_L(x,y) .$$
\subsubsection{The Doob Transform}
When $(M,g)$ is complete non compact and if the operator $L$ is non negative:
$$\forall \psi\in \cC^{\infty}_{0}(M): \int_{M} \left[|d\psi|_{g}^{2}-V\psi^{2}\right]\dm\ge 0$$
Then the Agmon-Allegretto-Piepenbrink theorem (\cite{Ag,FS,MP}) implies that there is a positive $W^{2,p}_{\loc}$ function $h$ such that 
$$Lh=0.$$
In particular the integration by part formula (see \ref{IPP}) yields that for any $\psi\in \cC^{\infty}_{0}(M)\colon$
\begin{equation}\label{IPP1}
\int_{M} \left[|d(h\psi)|_{g}^{2}-Vh^{2}\psi^{2}\right]\dm=\int_{M} |d\psi|_{g}^{2}h^{2}\dm.\end{equation}
Hence the \sch operator $L$ and the Laplacian $\Delta_{h^{2}m}$ are conjugates and we have the relations:
$$L(h\psi)=h\Delta_{h^{2}m}\psi$$
and
$$H_{L}(t,x,y)=h(x)h(y)H_{h^{2}m}(t,x,y).$$

\subsubsection{The Kato condition and uniform boundedness in $L^\infty$ }
The Laplacian $\Delta_{\m}$ is submarkovian that is to say:
$$\forall t>0, \forall x\in M\colon \int_{M}H_{\m}(t,x,y)\dm(y)\le 1.$$
An equivalent formulation is that 
$$\|e^{-t\Delta_{\m}}\|_{L^\infty\to L^{\infty}}\le 1.$$

We are interested in similar properties for \sch operators.
The non negativity of $L$ implies that the semi-group $\left(e^{-tL}\right)_{t}$ is uniformly bounded on $L^{2};$ but it is not necessary uniformly bounded  on $L^{\infty}.$
However the above Doob transformation guarantees that if the  \sch operator $L$ has a  zero eigenfunction $h$ satisfying 
$$1\le h\le \gamma$$ then the semi-group $\left(e^{-tL}\right)_{t}$ is uniformly bounded on $L^{\infty}.$ The study of Qi S.Zhang and  Z. Zhao (\cite{ZZ}, see also \cite{Zhao}) furnishes some equivalent properties for uniform boundedness in $L^{\infty}$ of the semigroup associated to \sch operator.
\begin{thm}\label{uniformbd} We say the  \sch operator $L=\Delta_{\m}-V$ is uniformly stable if one of the following equivalent condition is satisfied:
\begin{enumerate}[i)]	
\item $\displaystyle \sup_{t>0}\left\| e^{-tL}\right\|_{L^{\infty}\to L^{\infty} }<\infty.$
\item $\displaystyle \sup_{t>0}\left\| e^{-tL}\right\|_{L^{1}\to L^{1} }<\infty.$
\item There is a constant $\gamma$ such that for all $t>0$ and all $x\in M$:
$$\left(e^{-tL}1\right)(x)=\int_{M}H_{L}(t,x,y)\dm(y)\le \gamma.$$
\end{enumerate}
When $V$ is not identically zero, we say that the  \sch operator $L=\Delta_{\m}-V$ is gaugeable if one of the following equivalent condition is satisfied:
\begin{enumerate}[i)]	
\item There is $h\in W^{2,p}_{\loc}$ and $\gamma\ge 1$  such that 
$$Lh=0\ \,\mathrm{and}\,\  1\le h\le \gamma.$$
\item $L$ is subcritical, i.e. it has  a positive minimal Green kernel $G_{L}$ and there is a constant $C$ such that
$$\forall x\in M\colon\ \int_{M}G_{L}(x,y)V(y)\dm(y)\le C.$$
\end{enumerate}
Moreover, we have the following relations between these two properties: 
\begin{enumerate}[a)]
\item The gaugeability implies that uniform stability.
\item If  $\Delta_{\m}$ is stochastically complete, i.e. 
$\forall t>0\colon\ \left(e^{-t\Delta_{\m}}1\right)=1$, then the gaugeability condition is equivalent to the uniform stability.
\item If
the operator $\Delta_{\m}$ is subcritical and if the Kato constant of $V$ is smaller than $1$:
$$\Ka(V):=\sup_{x\in M}\int_{M}G_{\m}(x,y)V(y)\dm(y)<1,$$
then $L=\Delta_{\m}-V$ is gaugeable
\end{enumerate}

\end{thm}
\begin{rems}\label{remGaugeable}
\begin{enumerate}[a)]
\item The subcriticality of a \sch operator $L$ is a strengthening of the non negativity property. And we have the following equivalent properties:
\begin{enumerate}[i)]
\item $L$ is subcritical.
\item There is a non empty open set $\Omega\subset M$ and positive constant $\kappa$ such that 
$$\forall \psi\in \cC^{\infty}_{0}(M)\colon \kappa\int_{\Omega}\psi^{2}\dm\le \int_{M} \left[|d\psi|_{g}^{2}-V\psi^{2}\right]\dm.$$
\item For all relatively compact open subset $\Omega\subset M$, there is a positive constant $\kappa$ such that 
$$\forall \psi\in \cC^{\infty}_{0}(M)\colon \kappa\int_{\Omega}\psi^{2}\dm\le \int_{M} \left[|d\psi|_{g}^{2}-V\psi^{2}\right]\dm.$$
\item If $h\in W^{2,p}_{\loc}$ is a positive solution of the equation
$Lh=0$, then the operator $\Delta_{h^2m}$ is non parabolic.
\end{enumerate}
\item The set
$$\{\lambda\ge 0, \Delta_{\m}-\lambda V \mathrm{\ is\ gaugeable\ (resp.\ uniformly\ stable})\}$$ is an interval of the type
$[0,\omega)$ or $[0,\omega]$.
\item If $h\in W^{2,p}_{\loc}$ satisfies  
$Lh=0\ \,\mathrm{and}\,\  1\le h\le \gamma,$ then we have
$$ \Ka(V)\le \gamma-1.$$
\end{enumerate}

\end{rems}

\proof Let's explain why under the  stochastically completeness assumption, the uniform stability implies the gaugeability: 

The stochastically completeness condition implies that forall $x\in M$, the function $t\mapsto \int_{M}H_{L}(t,x,y)\dm(y)$ is non decreasing. Indeed, the semigroup properties implies that if $t,\tau>0$ then
$$\int_{M}H_{L}(t+\tau,x,y)\dm(y)=\int_{M\times M}H_{L}(t,x,z)H_{L}(\tau,z,y)\dm(z)\dm(y).$$
Using $H_{L}(\tau,z,y)\ge H_{\m}(\tau,z,y)$ and $\int_{M}H_{\m}(\tau,z,y)\dm(y)=1$, one gets:
$$\int_{M}H_{L}(t+\tau,x,y)\dm(y)\ge \int_{M}H_{L}(t,x,y)\dm(y).$$
Hence if the condition iii) is satisfied then we can define
$$h(x)=\sup_{t>0} \int_{M}H_{L}(t,x,y)\dm(y)=\lim_{t\to+\infty} \int_{M}H_{L}(t,x,y)\dm(y).$$
We have $1\le h\le \gamma$ and for all $\tau>0$:
$$\int_{M}H_{L}(\tau,x,y)h(y)\dm(y)=h(x).$$
Hence $Lh=0$.

About the equivalence between the two definitions of the gaugeability:\par
If we assume that property ii) holds, then 
$$h(x)=1+\int_{M}G_{L}(x,y)V(y)\dm(y)$$
defined a bounded solution of the equation $Lh=0$ moreover we have $h\ge 1$, hence the property i) holds.
If we assume that property i) holds, the Doob transform implies that $L$ is non negative. We have assumed that $V$ is not identically zero, hence  the non negativity of $L$ implies that $\Delta_{\m}$ is subcritical, now the Doob transform and the fact that $h$ is bounded guarantee that the operator $L$ is subcritical.
For a relatively compact domain $\Omega\subset M$, we consider the solution of the Dirichlet boundary problem:
$$\begin{cases}
\Delta_{\m}h_{\Omega}=Vh_{\Omega}&\ \mathrm{on}\ \Omega\\
h_{\Omega}=1&\ \mathrm{on}\ \partial\Omega
\end{cases}$$
If we introduce $G_{L}(\bullet,\bullet;\Omega)$ the Green function for the operator $L$ on $\Omega$ for the Dirichlet boundary condition
then we have 
$h_{\Omega}=1+v_{\Omega}$ where
$$v_{\Omega}=\int_{M}G_{L}(x,y;\Omega)V(y)dm(y).$$
The maximum principle implies that 
$$\frac{h}{\gamma} \le h_{\Omega}\le h,$$
and also that $\Omega\mapsto h_{\Omega}$ is increasing hence we can define
$$\widetilde h(x)=\lim_{\Omega\to M}h_{\Omega}(x)$$
and we will have
$$\widetilde h(x)\le \gamma$$ and 
$$\widetilde h(x)=1+\int_{M}G_{L}(x,y)V(y)dm(y).$$
Hence the property i) holds.
\endproof

\subsubsection{Elliptic estimates for \sch operators}
When the hypothesis iv) of the above \tref{uniformbd} is satisfied and when a Sobolev inequality holds, on can get estimate on the Green kernel of the operator $L$ or on sub-$L$-harmonic functions:
 \begin{thm}\label{sobolevgamma}Let  $(M,g,\m)$ be a weighted Riemannian manifold and assume it satisfies the 
Euclidean Sobolev inequality: $$\forall \psi\in \cC^\infty_{0}\left(M\right)\colon\ \mu \|\psi\|^{2}_{\frac{2n}{n-2}}\le \|d\psi\|^{2}_{2}.$$
Let  $L=\Delta_{\m}-V$ be a \sch operator and assume $L$ is gaugeable:
there is $h\in W^{2,p}_{\loc}$ and $\gamma\ge 1$ with
$$Lh=0\ \,\mathrm{and}\,\  1\le h\le \gamma.$$
Then the following properties hold:
\begin{enumerate}[i)]
\item The heat kernel associated to the Laplacian $\Delta_{L}$ satisfies
$$\forall x\in M,\forall t>0\colon \, H_{L}(t,x,x)\le \frac{c_n\gamma^{n}}{\left(\mu t\right)^{\frac n 2}}.$$
\item The associated positive minimal Green kernel satisfies:
$$\forall x,y\in M\colon \, G_{L}(x,y)\le \frac{c_n}{\mu ^{\frac n 2}}\frac{\gamma^{n}}{ d^{n-2}(x,y)}.$$
\item If $B(x,r)\subset M$ is a relatively compact geodesic balls in $M$ and if $v\in W^{1,2}_{\loc}(B(x,r))$ satisfies 
$$Lv\le 0$$ then for $p\ge 2$:
$$|v(x)|^{p}\le \frac{C(n,p)}{\left(\sqrt{\mu}\, r\right)^n} \gamma^{n-2+p}\int_{B(x,r)}|v|^p(y)\dm(y).$$
\end{enumerate}
\end{thm}
All these results follow from the Doob transform and the fact that the new measure $d\tilde \m=h^{2}\dm$ satisfies the Sobolev inequality:
$$\forall \psi\in \cC^\infty_{0}\left(M\right)\colon\ \mu\gamma^{-\frac 2n (n-2)} \|\psi\|^{2}_{\tilde \m,\frac{2n}{n-2}}\le \|d\psi\|^{2}_{\tilde \m,2}.$$

\subsubsection{Estimate on gaugeability constant}
In \cite{DevKato}, B. Devyver has studied the existence of such a function $h$ in the case where $V$ is not necessary non negative; we will use only the following result:
\begin{thm}Let $(M^{n},g,\m)$ be a complete weighted Riemanian manifold and $V\in L^{\infty}_{\loc}$ a non negative function.
Assume that the \sch operator $L=\Delta_{\m}-V$ is strongly positive: there is some $\delta>0$ such that the operator $\Delta_{\m}-(1+\delta)V$ is non negative:
$$\forall  \psi\in \cC^\infty_{0}\left(M\right)\colon(1+\delta)\int_{M}V\psi^{2}\dm\le \int_{M}|d\psi|_{g}^{2}\dm,$$
Assume moreover that the Kato constant of $V$ is small at infinity: there is a compact subset $K\subset M$ and some $\varepsilon\in (0,1)$ such that
$$\forall x\not\in K\colon\ \int_{M\setminus K}G_{\m}(x,y)V(y)\dm(y)\le 1-\varepsilon$$
Then there is a $h\in W^{2,p}_{\loc}$ and $\gamma\ge 1$ with
$$Lh=0\ \,\mathrm{and}\,\  1\le h\le \gamma.$$
\end{thm}

This results is an extension of a previous result of the author; in \cite{DevMA}, B. Devyver has shown:
\begin{thm}\label{KatoLp} Let  $(M,g,\m)$ be a weighted Riemannian manifold and assume it satisfies  
the Sobolev inequality: $$\forall \psi\in \cC^\infty_{0}\left(M\right)\colon\ \mu \|\psi\|^{2}_{\frac{2n}{n-2}}\le \|d\psi\|^{2}_{2}.$$
Let $V\in L^{\infty}_{\loc}$ be a non negative function such that
\begin{enumerate}[-]
\item for some $\varepsilon\in (0,1)$
$V\in L^{(1\pm \varepsilon)\frac n 2}$
\item $\ker_{L^{\frac{2n}{n-2}}}L=\left\{v\in L^{\frac{2n}{n-2}}(M,\dm): Lv=0\right\}=\{0\}.$
\end{enumerate}
 Then there is a $h\in W^{2,p}_{\loc}$ and $\gamma\ge 1$ with
$$Lh=0\ \,\mathrm{and}\,\  1\le h\le \gamma.$$
\end{thm}

For geometrical application, it is sometimes useful to obtain explicit bound on the function $h$ used in the Doob transform. The second hypothesis of the \tref{KatoLp} is satisfied when 
$$\int_{M}V^{\frac n2}\dm\le (1-\varepsilon)\mu$$ and in this case we can follow the argument given in \cite{DevMA} in order to get an estimate of $\|h\|_{\infty}$ that only  depends on $n,\mu, \varepsilon$, $\int_{M}V^{(1-\epsilon)\frac n2}\dm$ and $\int_{M}V^{(1+\epsilon)\frac n2}\dm$. 
The next result will give such a local estimate:
\begin{prop}Let  $(M^n,g,\m)$ be a weighted Riemannian manifold and  $B(x,2R)\subset M$ be a relatively compact geodesic ball and $V\in L^{\infty}_{\loc}$. Assume the following conditions
\begin{itemize}
\item The Sobolev inequality $\forall \psi\in \cC^\infty_{0}\left(B(x,2R)\right)\colon\ \mu \|\psi\|^{2}_{\frac{2n}{n-2}}\le \|d\psi\|^{2}_{2}.$
\item The strong positivity : there is a positive constant  $\delta>0$ such that
$$ \forall \psi\in \cC^\infty_{0}\left(B(x,2R)\right)\colon (1+\delta)\int_{B(x,2R)}V\psi^{2}\dm\le \int_{B(x,2R)}|d\psi|_{g}^{2}\dm,$$
\item A bound in the $L^p$-Kato class\footnote{here $G_{\m}(z,y)$ is the Dirichlet Green kernel of the Laplacien $\Delta_{\m}$ on $B(x,R)$.}  for $V$: $$ \Lambda^{p}=R^{2(p-1)}\sup_{z\in B(x,R)}\int_{B(x,R)} G_{\m}(z,y)V^{p}(y)\dm(y)$$
\end{itemize}
Then there is a constant $\gamma$ depending only on $n,p,\Lambda, \delta, \frac{m\left( B(x,2R)\right)}{\mu^{\frac n2}R^n}$ such that the solution of the Dirichlet boundary problem:
$$\begin{cases}\Delta_{\m} h-Vh=0&\mathrm{on}\ B(x,R)\\
h=1&\mathrm{on}\ \partial B(x,R)
\end{cases}$$
satisfies 
$$1\le h\le \gamma.$$
\end{prop}
\proof By scaling, we can suppose that $R=1$ and let $B:=B(x,1)$ and $2B:=B(x,2).$\par 
We first get an integral estimate on $v:=h-1$. If $ W^{1,2}_{0}(B)$ is the closure of $\cC^{\infty}_0(B)$ for the norm $\psi\mapsto \|d\psi\|_{2}+\|\psi\|_{2}$, we have $v\in W^{1,2}_{0}(B)$ and
$$\Delta_{\m} v-Vv=V$$ hence
$$\int_{B} \left[|dv|_{g}^{2}-Vv^{2}\right]\dm=\int_{B} Vv\dm\le \|v\|_{\infty}\int_{B} V\dm$$
We let $$L:=\|v\|_{\infty}.$$
Using the strong positivity and the function $$\xi(y)=\min\left\{2-d(x,y),1 \right\},$$ we get 
$$\int_{B} V\dm\le \int_{2B} V\xi^{2}\dm\le \frac{1}{1+\delta}\int_{2B} |d\xi|_{g}^{2}\dm\le m\left( 2B \right).$$
Using again  the strong positivity and the Sobolev inequality we get:
$$\frac{\mu\delta}{1+\delta}\left(\int_{B}v^{\frac{2n}{n-2}}\dm \right)^{1-\frac2n}\le \frac{\delta}{1+\delta}\int_{B} |dv|_{g}^{2}\dm\le \int_{B} \left[|dv|_{g}^{2}-Vv^{2}\right]\dm.$$
So that we get 
\begin{equation}\label{estimee1}
\left(\int_{B}v^{\frac{2n}{n-2}}\dm \right)^{1-\frac2n}\le L\, \frac{m\left( 2B\right) }{\mu \delta }.\end{equation}
The function $v$ is a solution of the integral equation:
\begin{equation}
\label{Intformula}
v(z)=\int_{B} G_{\m}(z,y)V(y)\dm(y)+\int_{B} G_{\m}(z,y)V(y)v(y)\dm(y).\end{equation}
Let $q=p/(p-1)$, using H\"older inequality and the integral estimate (ii-c) in \tref{sobolev}, we bounded the first term by
$$\int_{B} G_{\m}(z,y)V(y)\dm(y)\le \Lambda\left( \int_{B}G(z,y)\dm(y)\right)^{\frac 1q}\le \Lambda\left(\frac{m(B)}{\mu^{\frac n2}}\right)^{\frac{2}{nq}}.$$
Introducing  $$\mathbf{I}=\Lambda\left(\frac{m(2B)}{\mu^{\frac n2}}\right)^{\frac{2}{nq}},$$ we get
\begin{equation}\label{premestime}\int_{B} G_{\m}(z,y)V(y)\dm(y)\le \mathbf{I} \end{equation}
For the second term, using again the H\"older inequality, we get:
\begin{equation}\label{prem}\begin{split}
\int_{B} G_{\m}(z,y)V(y)v(y)\dm(y)&\le \left(\int_{B} G_{\m}(z,y)V^{p}(y)v(y)\dm(y)\right)^{\frac 1p} \left(\int_{B} G_{\m}(z,y)v(y)\dm(y)\right)^{\frac 1q}\\
& \le\Lambda L^{\frac 1p} \psi^{\frac 1q}  (z)\end{split}\end{equation} where
$$\psi(z):=\int_{B} G_{\m}(z,y)v(y)\dm(y).$$
If $\beta$ is such that 
$$\beta>\frac n2\ \mathrm{and}\ \beta\ge \frac{2n}{n-2}$$ then with $\alpha=\beta/(\beta-1)$ and the integral estimate (ii-c) in \tref{sobolev} we get:

\begin{equation*}
\begin{split}
\psi(z)&\le  \left(\frac{n}{(n-2)\alpha-n}\right)^{\frac1\alpha}\frac{1}{\mu}\left(m(B)\right)^{\frac{1}{\alpha}-1+\frac2n} \|v\|_{\beta}\\
&\le \left(\frac{n}{(n-2)\alpha-n}\right)^{\frac1\alpha}\frac{1}{\mu}\left(m(2B)\right)^{\frac{1}{\alpha}-1+\frac2n} \|v\|_{\beta}.
\end{split}\end{equation*}

The estimate (\ref{estimee1}) implies that:
$$\|v\|_{\beta}\le L^{1-\frac{n}{n-2}\frac{1}{\beta}}\, \left(\frac{m(2B)}{\delta\mu}\right)^{\frac{n}{n-2}\frac{1}{\beta}}.$$
After a bit of arithmetic, we get that:
\begin{equation}\label{estimee2}\int_{B} G_{\m}(z,y)V(y)v(y)\dm(y)\le\mathbf{I} \, L^{1-\frac{n}{n-2}\frac{1}{q\beta}}\, \left(\frac{m(2B)}{(\delta\mu)^{\frac n2}}\right)^{\frac{2}{n-2}\frac{1}{q\beta}} \left(\frac{n}{(n-2)\alpha-n}\right)^{\frac{1}{q\alpha}}.\end{equation}
With (\ref{estimee1}) and (\ref{estimee2}), we get
$$L\le \mathbf{I}+B^{\kappa}L^{1-\kappa}$$ where $\kappa=\frac{n}{n-2}\frac{1}{q\beta}$ and
$$B^{\kappa}=\mathbf{I} \, \, \left(\frac{m(2B)}{(\delta\mu)^{\frac n2}}\right)^{\frac{2}{n-2}\frac{1}{q\beta}} \left(\frac{n}{(n-2)\alpha-n}\right)^{\frac{1}{q\alpha}}.$$
In order to conclude, we distinguish two cases:
\begin{enumerate}[-]
\item The first one being when $\mathbf{I}\le \frac12$. Because the Kato constant of $V$ is smaller than $\mathbf{I}$, we know that $1\le h\le 2$.
\item The second case is  when $\mathbf{I}\ge \frac12$. The above inequality implies that 
$$L\le \max\{ 2\mathbf{I}, 2^{\frac1\kappa}B\}.$$
But 
$$2^{\frac1\kappa}B=c(n,p,\beta) \mathbf{I}^{q\beta\frac{n-2}{n}}\frac{m^{\frac 2n}(2B)}{\delta\mu}$$
Recall that by \tref{sobolev} iv), the Sobolev inequality implies that $\frac{m^{\frac 2n}(2B)}{\mu}$ is bounded from below by a constant that depends only of $n$ and that  $q\beta\frac{n-2}{n}>1$, hence there is a constant $c$ such that $B\ge c \mathbf{I}$ and we get:
$$L\le c(n,p,\beta) \mathbf{I}^{q\beta\frac{n-2}{n}}\frac{m^{\frac 2n}(2B)}{\min\{\delta,1\}\mu}.$$
\end{enumerate}

\endproof

\subsection{Agmon's volume estimate} In this subsection, we review some volume estimate that are implied by spectral estimates, these estimates are due to S. Agmon \cite{Ag} and P. Li and J. Wang \cite{LW1,LW2}).
The starting result is the following 
\begin{prop}\label{Agmon}
Assume $\m$ is a locally finite positive measure on $\R_{+}$. If we have the following spectral gap estimates:
\begin{equation}\label{Spg}\tag{SG}\forall \psi\in \cC^\infty_{0}\left(\R_{+}^{*}\right)\colon\  \frac{h^2}{4}\int_{\R_{+}}\psi^2(t)\dm(t)\le\int_{\R_{+}}\psi'^2(t)\dm(t),\end{equation}
then we have the dichotomy:
\begin{enumerate}[i)]
\item Either $m(\R_{+})<+\infty$ and $m\left([R,+\infty) \right)=\cO\left(e^{-hR}\right).$
\item Or $m(\R_{+})=+\infty$ and there is some positive constant $C$ such that for all $R\ge 1\colon$
$$ m\left([0,R]\right)\ge Ce^{hR}.$$
\end{enumerate}
\end{prop}
\begin{rem}\label{smooth} In the proof, we can always assume that $\m$ is smooth measure $$\dm=L(t)dt.$$
Indeed, if it is not the case, for $\rho$ a smooth non negative function with compact support in $(0,1)$ that satisfies 
$\int_{0}^{1}\rho(t)dt=1$, we can consider the  family of smooth measure
$\m_{\varepsilon}$ defined by
$$\m_{\varepsilon}(f)=\int_{\R_{+}\times \R} f(\tau-\varepsilon t)\rho(t)\dm(\tau)dt.$$
We have that $\m_{\varepsilon}$ weakly converge to $\m$ when $\epsilon\to0+$ and each $\m_{\varepsilon}$ satisfies the spectral gap inequality (\ref{Spg}) with the same constant.
\end{rem}

\proof
Assume that the measure $m$ is parabolic, that is to say there is a sequence of smooth function $\xi_{\ell}$ with compact support in $\R_{+}$ such that 
\begin{itemize}
\item $0\le \xi_{\ell}\le 1$
\item $\lim_{\ell\to+\infty}\xi_{\ell}(t)=1$
\item $\lim_{\ell\to+\infty}\int_{\R_{+}}\xi_{\ell}'^2(t)\dm(t)=0$
\end{itemize}
This implies that the spectral gap estimate (\ref{Spg}) is valid for any smooth function with support in $\R_{+}^{*}$ and constant outside some compact set, in particular:
$$\m(\R_{+})<+\infty.$$
We introduce the cut-off function:
 $$\xi(t)=\begin{cases} 0 &\mathrm{if}\ t\le 1 \\
t-1 &\mathrm{if}\ 1\le t\le 2\\
1 &\mathrm{if}\ 2\le t
\end{cases}$$
We test the spectral gap estimate (\ref{Spg}) for the function 
$$\psi_{R}(t)=\xi(t)\,e^{h\frac {\min\{t,R\}}{2}},$$ and when $R\ge 2$, we get the estimate
$$ \frac{h^2}{4}\int_{1}^{2} \psi^{2}_{R}(t)\dm(t)+\frac{h^2}{4}e^{hR}m\left([R,+\infty) \right)\le \int_{1}^{2} \psi'^{2}_{R}(t)\dm(t).$$
So that 
$$\frac{h^2}{4}e^{hR}m\left([R,+\infty) \right)\le (1+h)^{2}e^{2h}\,m\left([1,2]\right).$$

In the second case, we have that the measure $m$ is non-parabolic and necessary $$m\left(\R_{+}\right)=+\infty.$$
According to the \rref{smooth}, we can always assume that $\m$ is smooth: $\dm=L(t)dt.$
We introduce the function 
$$g(t)=\int_{t}^{\infty}\frac{ds}{L(s)},$$
The measure $m$ being  non-parabolic, we know that $g$ is well defined. Moreover
$g$ is a harmonic function for the Laplacian $$\Delta_{\m}=-\frac{1}{L(t)}\frac{d}{dt}L(t)\frac{d}{dt}.$$
 Hence the spectral gap estimate (\ref{Spg})  implies that
$$\int_{\R_{+}}g^2(t)\dm(t)<+\infty$$
Indeed if we test the spectral gap estimate (\ref{Spg}) for the function
$$g_{R}(t)=\begin{cases} \xi(t)\, \int_{t}^{R}\frac{ds}{L(s)}&\mathrm{if}\ t\le R\\
0 &\mathrm{if}\ t\ge R\end{cases}$$ then we get that for some constant $c$ independant of $R$:
$$\frac{h^2}{4}\int_{2}^{R} g^{2}_{R}(t)\dm(t)\le c+g_{R}(2).$$

Using the Doob transform with the function $g$ we get  the spectral gap estimate (\ref{Spg}):
$$\forall \psi\in \cC^\infty_{0}\left(\R_{+}^{*}\right)\colon\  \frac{h^2}{4}\int_{\R_{+}}\psi^2(t)g^{2}(t)\dm(t)\le\int_{\R_{+}}\psi'^2(t)g^{2}(t)\dm(t).$$
The new measure $g^{2}\dm$  is finite hence we already know that there is a constant $c$ such that for all $R\ge 0$ then
$$\int_{R}^{\infty} g^{2}(t)\dm(t)\le ce^{-hR}.$$
But for any function $\psi\in \cC^\infty_{0}\left(\R_{+}^{*}\right)$ we have
$$\int_{0}^{\infty} \left((\psi g)'\right)^{2}(t)\dm(t)=\int_{0}^{\infty} \left(\psi '\right)^{2}(t)g^2(t)\dm(t).$$
And choosing $\psi(t)=\xi(t-R+1)$ we get 
$$\int_{R}^{R+1} \left( g'\right)^{2}(t)\dm(t)\le \int_{R-1}^{\infty} g^2(t)\dm(t)\le ce^{-hR}.$$
So that 
$$\int_{R}^{R+1} \frac{ds}{L(s)}\le  ce^{-hR}.$$
But the Cauchy-Schwarz inequality yields 
$$1\le \left(\int_{R}^{R+1} \frac{ds}{L(s)}\right)\left(\int_{R}^{R+1} L(s)ds\right) \le ce^{-hR}m\left([R,R+1]\right).$$
\endproof
Let's now give some classical consequence of this inequality that are borrowed from \cite{LW1}.

\begin{cor} Let $(M,g,\m)$ be a complete weighted Riemanian manifold and $K\subset M$ be a compact set and $\cU\subset M\setminus K$ be an unbounded connected compounent of $M\setminus K$ that satisfies the spectral gap 
 and $$\forall \psi\in \cC^\infty_{0}\left(\cU\right)\colon\  \frac{\lambda_{0}^2}{4}\int_{\cU}\psi^2\dm\le\int_{M}|d\psi|^2\dm.$$
Assume that $f\colon \cU\rightarrow \R$ satisfies:
 $$\Delta_{\m} f\le \lambda f$$ where
 $\lambda<\lambda_{0}$.
and let $h=2\sqrt{\lambda-\lambda_{0}}$. We have the dichotomy:
\begin{enumerate}[i)]
\item Either $f\in L^{2}$ and when $R\to+\infty$: $\int_{\cU\setminus B(o,R)} f^{2}\dm =\cO\left(e^{-hR}\right).$
\item There is some positive constant $C$ such that for all $R\ge 1\colon$
$$ \int_{\cU\cap B(o,R)} f^{2}\dm \ge Ce^{hR}.$$
\end{enumerate}
\end{cor}
\proof We test the above spectral gap for radial function
$$\psi(x)=f(d(K,x))$$ and for the measure
$$\mu([0,R])=\m\left(\{x\in \cU, d(x,K)<R\}\right),$$ we get
$$\forall \psi\in \cC^\infty_{0}\left(\R^{*}_{+}\right)\colon\  \frac{h^2}{4}\int_{\R_{+}}f^2(t)d\mu \le\int_{\R_{+}}f'(t)^{2}d\mu$$
The corollary is then a direct consequence of the \pref{Agmon}.
\endproof

Also a logarithmic change of variable yields the following consequence of a Hardy type inequality:
\begin{prop}\label{hardyvol}
Assume $m$ is a locally finite positive measure on $[1,\infty)$. If we have the following Hardy type inequality:
\begin{equation*}\forall \psi\in \cC^\infty_{0}\left( (1,\infty)\right)\colon\  \frac{(\nu-2)^2}{4}\int_1^{+\infty}\frac{\psi^2(t)}{t^{2}}\dm(t)\le\int_1^{+\infty}\psi'^2(t)\dm(t),\end{equation*}
then we have the dichotomy:
\begin{enumerate}[i)]
\item Either $\int_1^{+\infty}\frac{\dm(t)}{t^{2}}<+\infty$ and when $R\to+\infty$:
$$\int_R^{+\infty}\frac{\dm(t)}{t^{2}}=\cO\left(\frac{1}{R^{\nu-2}}\right).$$
\item Or  there is some positive constant $C$ such that for all $R\ge 1\colon$
$$\int_{[1,R]}\dm(t)\ge C R^{\nu}.$$
\end{enumerate}
\end{prop}

\subsection{Asymptotic of the Green kernel}\label{greenasymp}
\subsubsection{Near the pole}\label{nearpole}Let $L=\Delta_{\m}-V$ be a \sch operator on a weighted smooth Riemannian manifold $(M,g,\m)$ of dimension $n>2$. If $G$ is a positive solution of the equation\marginpar{a reverifier}
$$LG=\delta_o,$$ then according to \cite[section 17.4]{HoIII} $G$ has polyhomogeneous expansion near $o$ whose first term is $$G(x)\simeq \frac{c(n)}{d^{n-2}(x,o)},$$ where 
$c_{n}=\left((n-2)\mathrm{vol} \,\bS^{n-1}\right)^{-1}.$ More precisely, if we let $r(x):=d^{n-2}(x,o)$, then there is $\psi\in C^{1}(M)$ such that:
$$G=\frac{\psi}{r^{n-2}}\ \mathrm{and} \ \psi(o)=c_n. $$
In particular if we define $b\colon M\rightarrow \R_+$ by 
$G=c_{n}b^{2-n}$ then
$b(o)=0$ and $$|db|(x)=1+\cO\left( d(o,x)\right).$$
\subsubsection{Near infinity} We consider non parabolic weighted Riemannian manifold $(M,g,\m)$ and $G_{\m}$ will be its minimal positive Green kernel.
If $o\in M$ and $K$ is compact subset  of $M$ containing $o$ in its interior then
$$\int_{M\setminus K} |d_{x}G_{\m}(o,x)|^{2}\dm(x)<\infty.$$
Indeed, we can always assume that the boundary of $K$ is smooth.
When $\Omega$ is a relatively compact open subset of $M$ containing $K$, we will note $G_{\m}^{\Omega}$ the  minimal positive Green kernel of  $(\Omega,g,\m)$. We know that
$$\lim_{\Omega\to M}G_{\m}^{\Omega}(o,x)=G_{\m}(o,x)$$
where the convergence is in $\cC^{\infty}(M\setminus \{o\})$.
But the Green formula yields that 
$$\int_{M\setminus K} |d_{x}G_{\m}^{\Omega}(o,x)|^{2}\dm(x)=\int_{\partial K}G_{\m}^{\Omega}(o,x)\frac{\partial G_{\m}^{\Omega}}{\partial\vec \nu_{x}}(o,x)d\sigma(x)$$ where $\vec\nu\colon\partial K\rightarrow TM$ is the unit normal inward normal to $K$.
Hence 
$$\int_{M\setminus K} |d_{x}G_{\m}(o,x)|^{2}\dm(x)\le \int_{\partial K}G_{\m}(o,x)\frac{\partial G_{\m}}{\partial\vec \nu_{x}}(o,x)d\sigma(x)$$ 
We are now interested in the equality in the above formula:
\begin{prop}Assume that $\lim_{x\to \infty}G(o,x)=0$ then 
\begin{equation}
\label{GreenIPP}\int_{M\setminus K} |d_{x}G_{\m}(o,x)|^{2}\dm(x)= \int_{\partial K}G_{\m}(o,x)\frac{\partial G_{\m}}{\partial\vec \nu_{x}}(o,x)d\sigma(x).
\end{equation}
Moreover the measure $G_{\m}(o,x)^{2}\dm(x)$ is parabolic on $M\setminus K$.
\end{prop}
\proof Let $\ell>0$, our hypothesis implies that the set $\left\{x\in M, G_{\m}(o,x)\le \frac 1\ell\right\}\cup\{o\}$ is compact.
Let $u$ be a smooth function on $\R_{+}$ such that 
$|u'| \le 2$, $u=0$ on $[0,1]$ and $u=1$ on $[2,+\infty)$.
We introduce the cut-off function defined by:
$$\xi_{\ell}(x)=u\left(\ell G_{\m}(o,x)\right)$$ 
Let $\epsilon:=\inf_{x\in \partial K}G_{\m}(o,x)$. If $\ell\epsilon>1$ then the maximum principle guarantees the inclusion:
$$\left\{x, G_{\m}(o,x)\le \frac 1\ell\right\}\subset M\setminus K.$$
And the Green formula yields:
\begin{equation*}
\begin{split}
\int_{M\setminus K}\left|d_{x}\left(\xi_{\ell}(x)G_{\m}(o,x)\right)\right|^{2}\dm(x)&=\int_{\partial K}G_{\m}(o,x)\frac{\partial G_{\m}}{\partial\vec \nu_{x}}(o,x)d\sigma(x)\\
&\hspace{1cm}+\int_{M\setminus K}\left|d\xi_{\ell}\right|^2 G_{\m}^{\Omega}(o,x)^{2}\dm(x).
\end{split}
\end{equation*}
Define $$\Omega_\ell=\left\{x,\in M\frac 1\ell\le G_{\m}(o,x)\le \frac{2}{\ell}\right\}$$
\begin{equation*}
\begin{split}\int_{M\setminus K}\left|d\xi_{\ell}\right|^2 G_{\m}(o,x)^{2}\dm(x)&
\le 4\ell^{2}\int_{\Omega_\ell}\left|d_{x}G_{\m}(o,x)\right|^2 G_{\m}(o,x)^{2}\dm(x)\\
&\le 16 \int_{\Omega_\ell}\left|d_{x}G_{\m}(o,x)\right|^2 \dm(x)
\end{split}
\end{equation*}
Hence
$$\lim_{\ell\to +\infty}\int_{M\setminus K}\left|d\xi_{\ell}\right|^2 G_{\m}(o,x)^{2}\dm(x)=0$$ and the equality (\ref{GreenIPP}). Moreover the sequence $(\xi_\ell)$ satisfies the require properties (\ref{parabolic}) that show the parabolicity of the measure $G_{\m}(o,x)^{2}\dm(x)$ on $M\setminus K$.
\endproof

With the Doob transform, we have a similar result for \sch operator :
\begin{prop}\label{Greenpara}We consider non parabolic weighted Riemannian manifold $(M,g,\m)$ and $V$ a locally bounded non negative function. Assume that the \sch operator $L=\Delta-V$ is gaugeable and that for some  $p\in m$, the Green kernel of $L$ satisfies $\lim_{x\to \infty}G_L(p,x)=0$ then the measure $G_{L}^{2}(p,y)\dv_{g}(y)$ is parabolic on $M\setminus B(o,2)$.
\end{prop}

A last but useful property of the Green kernel, is the following universal Hardy type inequality \cite{CarHardy}:
\begin{prop}\label{Hardy}We consider non parabolic weighted Riemannian manifold $(M,g,\m)$ of dimension $n>2$. If $o\in M$, we let $b(x)=G_{\m}(o,x)^{-\frac{1}{n-2}}$, then
$$\forall \psi \in \cC^\infty_{0}(M)\colon\ \frac{(n-2)^{2}}{4}\int_{M} \frac{ |db|^2}{b^2} \psi^2\dm\le \int_M
|d\psi|^2\dm.$$
\end{prop}
In fact  when $g$ is  any positive harmonic function, then the above inequality holds for $b=g^{-\frac{1}{n-2}}.$
\subsection{Some formula for the gradient of the Green kernel} The inequalities given in this subsection are in fact due to T. Colding and W. Minicozzi (\cite{Colding,ColdingMini}).
When $g$ is a positive harmonic function on a  Riemannian manifold $(M^{n},g)$, then the Yau's inequality \cite[lemma 2]{Yau_CPAM} implies that 
\begin{equation}\label{yau1}
\Delta |dg|^{\frac{n-2}{n-1}}-\frac{n-2}{n-1}\ricm |dg|^{\frac{n-2}{n-1}}\le 0.
\end{equation}
But if we define $b$ by $$g=\frac{1}{b^{n-2}}.$$
Then we have 
$$|dg|^{\frac{n-2}{n-1}}={ (n-2)^{\frac{n-2}{n-1}}}g\left|db\right|^{\frac{n-2}{n-1}}.$$
The Doob transform yields that $u:=\left|db\right|^{\frac{n-2}{n-1}}$ satisfies
$$\Delta_{g^{2}}u\le \frac{n-2}{n-1}\ricm u.$$
Hence (see \ref{chainrule}), for all $\alpha\ge 1$ we get:
$$\Delta_{g^{2}}u^{\alpha}\le \alpha\frac{n-2}{n-1}\ricm u^{\alpha}.$$
Using the Doob transform again, we get that for all $p\ge \frac{n-2}{n-1}:$
\begin{equation}\label{yau2}\Delta \left(g|db|^{p}\right)\le p\ricm  \left(g|db|^{p}\right).\end{equation}
\subsection{An elliptic estimate}
In this subsection, we obtain a new gradient estimate for the gradient of positive harmonic function; our result is based on a new variation on 
 the De Giorgi-Nash-Moser iteration scheme.
 
 \begin{prop}\label{estimedb} Let $(M^n,g)$ be a Riemannian manifold that satisties
 \begin{enumerate}[-]
 \item the Sobolev inequality:$$\forall \psi\in \cC^\infty_{0}\left(M\right)\colon\ \mu \|\psi\|^{2}_{\frac{2n}{n-2}}\le \|d\psi\|^{2}_{2}.$$
 \item Gaugeability: there is a function $h\colon M\longrightarrow [1,\gamma]$ such that
 $$\Delta h-\frac{n-2}{n-1}\ricm h=0.$$
 \end{enumerate}
Consider $g\colon M\longrightarrow \R_{+}^{*}$  a positive harmonic function and let  $b\colon M\longrightarrow \R_{+}^{*}$ be defined by 
$$g=\frac{1}{b^{n-2}}.$$
Assume moreover  $R>0$ is such that the set 
$\Omega_{R}^{\#}=\left\{x\in M; \frac R 2\le b(x)\le \frac 5 2 R\right\}$ is compact and let
$\Omega_{R}=\left\{x\in M; R \le b(x)\le  2 R\right\}$. 
 Then if $p>n$:
 $$\sup_{\Omega_{R}} |db|^{p-n}\le \frac{C_n^{1+p}\gamma^{p\frac{n-1}{n-2}+n-2}}{\mu^{\frac n2}R^{n}}\int_{\Omega_{R}^{\#}} |db|^{p}\dv_{g}.$$
 \end{prop}
\begin{rem}
 The second hypothesis is satisfied when we have the following bound on the Kato constant of the Ricci curvature:
$$\Ka\left(\ricm\right)\le \frac{n-1}{n-2}\,\left(1-\frac1\gamma\right).$$
\end{rem}
\proof We let 
$$f:= g\left|db\right|^{\frac{n-2}{n-1}}\ ;$$
The Yau's inequality (\ref{yau1}) implies that 
$$\Delta f-\frac{n-2}{n-1}\ricm f\le 0.$$
Hence the function 
$F=f/h$ satisfies
$\Delta_{h^{2}}F\le 0$ and for all $\alpha\ge 1$ we have:
$$\Delta_{h^{2}}F^{\alpha}\le 0.$$
So that if $\xi\in \cC^{\infty}_{0}(\Omega_{R}^{\#})$ we have:
\begin{equation}
\label{un}	
\int_{M}\left| d\left( \xi F^{\alpha}\right)\right|^2h^2\dv_{g}\le \int_{M}\left| d\xi \right|^2 F^{2\alpha}h^2\dv_{g}.
\end{equation}
Moreover the Sobolev inequality and the assumed bound on $h$ implies that for $\widehat \mu:=\mu \gamma^{\frac 4 n-2}$, we have:
\begin{equation}\label{deux}\widehat \mu \left(\int_{M}\left( \xi F^{\alpha}\right)^{\frac{2n}{n-2}}h^2\dv_{g}\right)^{1-\frac 2n}\le \int_{M}\left| d\left( \xi F^{\alpha}\right)\right|^2h^2\dv_{g}.\end{equation}
We define now $\dm=h^{2}\dv_{g}$, $\kappa:=\frac{n}{n-2}$,
$$R_{k}=2R+\sum_{\ell=k}^{\infty}\frac{R}{2^{\ell+2}}\ \mathrm{and}\ r_{k}=R-\sum_{\ell=k}^{\infty}\frac{R}{2^{\ell+2}}\ \mathrm{and}\ \Omega_{k}=\{b\in  [r_{k},R_{k}]\}.$$
We are going to use the inequalities (\ref{un},\ref{deux}) with
   $$\xi_{k}=\rho_{k}(b)$$ where
$$\rho_{k}=\begin{cases}
1&\ \mathrm{on}\ [r_{k+1},R_{k+1}]\\
0&\ \mathrm{outside} \ [r_{k},R_{k}]
\end{cases}$$ and 
$$\left|\rho_{k}'\right|\le \frac{2^{k+2}}{R}.$$ And we will get 
\begin{equation*}
\widehat \mu \left(\int_{\Omega_{k+1}}\left(  F^{\alpha}\right)^{\frac{2n}{n-2}}\dm\right)^{1-\frac 2n}\le \frac{4^{k+2}}{R^{2}}\int_{\Omega_{k}} |db|^2F^{2\alpha}\dm.\end{equation*}
But
$$|db|^2F^{2\alpha}=F^{2\alpha+2\frac{n-1}{n-2}} \left(\frac{g}{h}\right)^{-2\frac{n-1}{n-2}}.$$
On $\Omega_{k}$ we have:
$$ \left(\frac{g}{h}\right)^{-2\frac{n-1}{n-2}}\le \gamma^{2\frac{n-1}{n-2}}\left(\frac{5R}{2}\right)^{2n-2}.$$
We introduce now:
$$\beta_{k+1}=\kappa2\alpha_{k}\ \mathrm{and}\ \beta_{k}=2\alpha_{k}+2\frac{n-1}{n-2}$$ where
$$\beta_{k}=\kappa^{k}\left(\beta_{0}-n\frac{n-1}{n-2}\right)+n\frac{n-1}{n-2}.$$
And we get
\begin{equation}\label{trois}
 \left(\int_{\Omega_{k+1}}  F^{\beta_{k+1}}\dm\right)^{1-\frac 2n}\le \frac{\gamma^{2\frac{n-1}{n-2}} 4^{k+2}}{\widehat\mu R^{2}} \left(\frac{5R}{2}\right)^{2n-2}\int_{\Omega_{k}} F^{\beta_{k}}\dm.\end{equation}
 For $p=\frac{n-2}{n-1}\beta_{0}$, by iteration we get

$$\lim_{k\to\infty}\left(\int_{\Omega_{k+1}}  F^{\beta_{k+1}}\dm\right)^{\kappa^{-n}}\le \Gamma \int_{\Omega_R^{\#}}  F^{2\alpha_0}\dm.$$
where $$\Gamma=\left(\frac{16 \gamma^{2\frac{n-1}{n-2}} }{\widehat\mu R^{2}} \left(\frac{5R}{2}\right)^{2n-2}\right)^{\sum_{\ell=0}^{\infty}\kappa^{-\ell}}4^{\sum_{\ell=0}^{\infty}\ell\kappa^{-\ell}}.$$
But
\begin{equation*}\begin{split}
\lim_{k\to\infty}\left(\int_{\Omega_{k+1}}  F^{\beta_{k+1}}\dm\right)^{\kappa^{-n}}
&=\sup_{\Omega_{R}} F^{\beta_0-n\frac{n-1}{n-2}}\\
&=\sup_{\Omega_{R}} h^{-(p-n)\frac{n-1}{n-2}}b^{-(p-n)(n-1)}|db|^{p-n}\\
&\ge \gamma^{-(p-n)\frac{n-1}{n-2}}(2R)^{-(p-n)(n-1)} \sup_{\Omega_{R}}|db|^{p-n}
\end{split}
\end{equation*}
and 
$$\Gamma=c(n)\gamma^{n\frac{n-1}{n-2}} \widehat\mu^{-\frac n2} R^{n(n-2)}.$$
Moreover
\begin{equation*}\begin{split}
\int_{\Omega_R^{\#}}  F^{2\alpha_0}\dm&=\int_{\Omega_R^{\#}}  h^{-p\frac{n-1}{n-2}}b^{-p(n-1)}|db|^{p}h^{2}\dv_{g}\\
&\le \left(\frac{2}{R}\right)^{p(n-1)} \int_{\Omega_R^{\#}}  |db|^{p}\dv_{g}.
 \end{split}
\end{equation*}

Hence after a bit of arithmetic, we obtain:
$$\sup_{\Omega_{R}} |db|^{p-n}\le \frac{c(n)4^{p}\gamma^{p\frac{n-1}{n-2}+n-2}}{ \mu^{\frac n2} R^n} \int_{\Omega_R^{\#}}  |db|^{p}\dv_{g}.$$
\endproof
\section{Case of compact manifold} In this section, we are going to elaborate from a recent result of Qi S. Zhang and M. Zhu \cite{ZZhu1} in order to obtain geometric and topological estimates based on a Kato bound for the Ricci curvature.
\subsection{A differential inequality}
In \cite{ZZhu1}, the authors have shown the following:
\begin{prop}Let $(M^n,g)$ be a {\it complete} Riemannian manifold and $u\colon [0,T]\times M\rightarrow \R$ be a positive solution of the heat equation:
$$\frac{\partial u}{\partial t} +\Delta u=0$$ and $J\colon [0,T]\times M\rightarrow \R$ a auxiliary positive {\it gauging} function. The function 
$$Q:=\alpha J| d\log u|^{2}-\frac{\partial}{\partial t} \log u$$ satisfies
\begin{equation*}
\begin{split}
\left(\frac{\partial}{\partial t} +\Delta \right)Q-2\la d\log u,dQ\ra&\le \alpha|d\log u|^{2}\left(\frac{\partial J}{\partial t} +\Delta J+\frac{5}{\delta}\frac{ |dJ|^{2}}{J}-2\ricm J\right)\\
&-(2-\delta)\alpha J |\nabla d\log u|^{2}+\delta\alpha J |d\log u|^{4}
\end{split}\end{equation*}
\end{prop}

\subsection{Finding a good gauge function} We are now looking for a solution of the equation
\begin{equation}\label{equaJ}\left\{
\begin{array}{l}
 \frac{\partial}{\partial t} J+\Delta J+\frac{5}{\delta}\frac{ |dJ|^{2}}{J}-2\ricm J=0\\
J(0,x)=1
\end{array}\right.\end{equation}
If we let 
$$I:=J^{-\frac{5-\delta}{\delta}}\ \mathrm{or}\ J=I^{-\frac{\delta}{5-\delta}},$$ this equation is equivalent to 
$$\left\{
\begin{array}{l}
 \frac{\partial}{\partial t} I+\Delta I=2\frac{5-\delta}{\delta}\ricm I=0\\
I(0,x)=1
\end{array}\right.$$
Using the Duhamel's formula, this equation can be converted into the integral equation:
$$I(t,x)=1+2\frac{5-\delta}{\delta}\int_{0}^{t} \int_{M}H(t-s,x,y)\ricm(y)I(s,y)\dv_{g}(y)ds.$$
We recall the definition of the (parabolic) Kato constant of the function $\ricm$:
$$\ka_{T}(\ricm)=\sup_{x\in M}\int_{0}^{T}\int_{M}H(t,x,y)\ricm(y)\dv_{g}(y)dt=\left\| \int_{0}^{T}e^{-t\Delta}\ricm dt\right\|_{\infty}.$$
An easy application of the fixed point theorem in $L^{\infty}([0,T]\times M)$ yields that if
$\delta\in (0,1)$ and $\ka_{T}(\ricm)\le \frac{\delta}{16}$ then the above integral equation has a unique solution $I\in L^{\infty}([0,T]\times M)$ with
$$1\le I(t,x)\le 1+4\frac{5-\delta}{\delta}\ka_{T}(\ricm)\le e^{4\frac{5-\delta}{\delta}\ka_{T}(\ricm)}.$$
Hence
\begin{lem}\label{J}The equation (\ref{equaJ}) has a unique solution $J$ and this solution satisfies:
$$e^{-4\ka_{T}(\ricm)}\le J(t,x)\le 1.$$
\end{lem}
\subsection{Li and Yau's gradient estimates}
If we assume now that $M$ is compact and if we assume that the function $tQ$ reachs its maximum on $[0,T]\times M$ at $(t_{0},x_{0})$. Using the gauging function $J$ given by the (\lref{J}), we get at this point:
\begin{equation*}\begin{split}
\frac{Q(t_{0},x_{0})}{t_{0}}&\le \left(\frac{\partial}{\partial t} +\Delta \right)Q-2\la d\log u,dQ\ra\\
&\le-(2-\delta)\alpha J |\nabla d\log u|^{2}+\delta\alpha J |d\log u|^{4}.
\end{split}
\end{equation*}
Let $\alpha\in (0,1)$ and
assume $Q(t_{0},x_{0})\ge 0$, we get at $(t_{0},x_{0})$:
\begin{equation}\begin{split}
|\nabla d\log u |^{2}\ge \frac1 n\left(\Delta \log u\right)^{2}&=\frac1n\left(|d\log u|^{2}-\frac{\partial}{\partial t}|og u\right)^{2}\\
&=\frac1n\left(Q+(1-\alpha J)|d\log u|^{2}\right)^{2}\\
&\ge \frac 1n Q^{2}+\frac{(1-\alpha J)^{2}}{n}|d\log u|^{4}
\end{split}
\end{equation}
So that 
\begin{equation*}\begin{split}0\le&\frac{Q(t_{0},x_{0})}{t_{0}}\left(1-\frac{(2-\delta)\alpha J}{n} t_{0}Q(t_{0},x_{0})\right)\\
&\hspace{1cm}+\left(\delta-(2-\delta)\frac{(1-\alpha J)^{2}}{n}\right)\alpha J|d\log u|^{4}.\end{split}
\end{equation*}
Because $0\le J\le 1$, we have
$$(1-\alpha J)^{2}\le (1-\alpha )^{2}$$ 
If we assume $$\delta< \frac{2}{n+1}$$ then we choose $\alpha=1-\sqrt{\frac{ n\delta}{2-\delta}}$ and get:
$$ t_{0}Q(t_{0},x_{0})\le \frac{n}{(2-\delta)\alpha J}.$$
We make now several choices:
If we assume that $\ka_{T}(\ricm)\le \frac{1}{16n}$ then we let
$\delta=16\ka_{T}(\ricm)$ and $\alpha=1-\sqrt{\frac{ n\delta}{2-\delta}}$. With these choices, we have
$$\alpha J\ge e^{-4\ka_{T}}\left(1-4\sqrt{n\ka_{T}}\right)\ge e^{-8\sqrt{n\ka_{T}}}.$$ and 
$$(2-\delta)\alpha J\ge 2\left(1-\frac{\delta}{2}\right)\alpha J\ge 2e^{-12\ka_{T}- 4\sqrt{n\ka_{T}}}.$$
But 
$$12\ka_{T}+ 4\sqrt{n\ka_{T}}\le 4\sqrt{\ka_{T}}\left(\frac{3}{4\sqrt{n}}+\sqrt{n}\right)\le 8\sqrt{n\ka_{T}}.$$

Finally, we have shown:
\begin{prop}\label{LYI}Assume that $(M^{n},g)$ is a compact Riemannian manifold such that for some $T>0$:
$$\ka_{T}=\sup_{x\in M}\int_{0}^{T}\int_{M}H(t,x,y)\ricm(y)\dv_{g}(y)dt\le \frac{1}{16n}$$
If $u\colon [0,T]\times M\rightarrow \R$ be a positive solution of the heat equation:
$$\frac{\partial u}{\partial t} +\Delta u=0$$ then we have
 on $[0,T]\times M$:
 $$e^{-8\sqrt{n\ka_{T}}}\,\frac{|du|^{2}}{u^{2}}-\frac{1}{u}\frac{\partial u}{\partial t} \le \frac{n}{2t}e^{8\sqrt{n\ka_{T}}}$$
 and
 $$e^{-2}\,\frac{|du|^{2}}{u^{2}}-\frac{1}{u}\frac{\partial u}{\partial t} \le \frac{n}{2t}e^{2}$$
\end{prop}
\subsection{Heat kernel estimate}

For the remaining part of this section, we assume that $(M^{n},g)$ is a compact Riemannian manifold of diameter
$$D:=\diam (M,g)$$ and we define
$T(M,g)$ the largest times $T$ such that 
$$\ka_{T}(\ricm)\le \frac{1}{16n}$$ and we introduce the scaled invariant geometric quantity $\xi(M,g)$:
$$\xi^{2}(M,g)T=D^{2}.$$
For instance, if we have 
$\ricci_{g}\ge -(n-1)\kappa^{2}g$ then
$$\xi(M,g)\le 4n\kappa D.$$
We also introduce:
$$\nu:=e^{2}n.$$
And following the arguments of P. Li and S-T.Yau \cite{LY}, we easily get:
\begin{lem}\label{harnack}Let $u\colon [0,T]\times M\rightarrow \R$ be a positive solution of the heat equation. 
If $s\le t\le T$ and $x,y\in M$:
$$u(s,x)\le \left(\frac t s\right)^{\frac\nu2}u(t,x)$$
and
$$u(s,x)\le  \left(\frac t s\right)^{\frac\nu2}e^{2\frac{d^{2}(x,y)}{t-s}}\,u(t,y)$$
\end{lem}
\proof The first assertion is a direct consequence of the \pref{LYI}, indeed we have
$$e^{-2}\,\frac{|du|^{2}}{u}\le\frac{\partial u}{\partial t} + \frac{\nu}{2t}u=t^{-\frac\nu2}\frac{\partial}{\partial t}\left(t^{\frac\nu2}u\right).$$
Concerning the second statement , we introduce $\gamma\colon[0,t-s]\rightarrow M$ a minimizing geodesic joining $y$ to $x$ and we define:
$$\phi(\tau)=\log u\left(t-\tau,\gamma(\tau)\right),$$
so that 
$\phi(0)=\log u\left(t,y\right)$ and $\phi(t-s)=\log u\left(s,x\right)$
we have
\begin{equation}
\begin{split}
\dot\phi(\tau)&=-\frac{1}{u}\frac{\partial u}{\partial t}+\la \dot\gamma,du\ra\\
&\le\frac{\nu}{2(t-\tau)}-e^{-2}\frac{|du|^{2}}{u^{2}}+\la \dot\gamma,du\ra\\
&\le\frac{\nu}{2(t-\tau)}+\frac{e^{2}}{4}|\dot\gamma|^{2}=\frac{\nu}{2(t-\tau)}+\frac{e^{2}d^{2}(x,y)}{4(t-s)^{2}}\\
&\le\frac{\nu}{2(t-\tau)}+2\frac{d^{2}(x,y)}{(t-s)^{2}}
\end{split}
\end{equation}
Integrating this, we get 
$$\frac{u(s,x)}{u(t,y)}\le \left(\frac t s\right)^{\frac\nu2}e^{2\frac{d^{2}(x,y)}{t-s}}.$$
\endproof

This result leads to heat kernel bound:
\begin{thm}\label{heatnu} There is a constant $c_{n}$ such that for 
 $0\le s\le t\le T/2$ and $y\in B(x,\sqrt{t})$ then
$$H(s,y,y)\le\left( \frac t s\right)^{\frac \nu 2} \frac{c_{n}}{\vol B(x,\sqrt t)}.$$
Moreover for any $s\ge T/2$ and $x,y\in M$, we have :
$$ H(s,x,y)\le \frac{c_{n}^{1+\xi}}{\vol M}.$$
\end{thm}
\proof
Let $\gamma_{n}=2^{\frac\nu 2}e^{2}.$
Using the \lref{harnack}, we know that if $d(x,y)\le \sqrt{t}$ and $t\le T/2$, then 
$$H(t,x,x)\le \gamma_{n} H(2t,x,y)\le\gamma_{n}^{2} H(3t,x,x).$$
But the function $t\mapsto H(t,x,x)$ is non increasing, hence
$$ \gamma_{n}^{-1}H(t,x,x)\le H(2t,x,y)\le \gamma_{n}   H(t,x,x).$$
And also
$$ \gamma_{n}^{-2}H(t,x,x)\le H(t,y,y)\le \gamma_{n}^{2}   H(t,x,x).$$
Integrating the inequality: $H(t,x,x)\le \gamma_{n} H(2t,x,y)$ 
 over $y\in B(x,\sqrt{t})$ and using that 
$$\int_{B(x,\sqrt{t})}H(2t,x,y)\dv_{g}(y)\le \int_{M}H(2t,x,y)\dv_{g}(y)=1,$$
we get
$$H(t,x,x)\le \frac{\gamma_{n}}{\vol B(x,\sqrt{t})}.$$
And for $y\in B(x,\sqrt{t})$ 
$$H(t,y,y)\le \frac{\gamma_{n}^{3}}{\vol B(x,\sqrt{t})}.$$
The first part of the \lref{harnack} implies the first assertion.

Concerning the second assertion. Let  $t\le T/2$ and let $y,z\in M$ be such that $d(z,y)\le \sqrt{t}$ , then for any $\sigma\ge 0$:
$$H(\sigma+t,x,y)\le \gamma_{n} H(\sigma+2t,x,z).$$
Assume now $s\ge T/2$.

If $D\le \sqrt{T/2}$, using $t=D^{2}$  and $\sigma=s-t$,
we get for all $y,z\in M$:
$$H(s,x,y)\le \gamma_{n} H(s+t,x,z).$$ Integrating this inequality over $z\in M$ and get 
$$H(s,x,y)\le \frac{\gamma_{n}}{\vol M}.$$
Assume now $\sqrt{T/2}\le D$ and let $N\in \N$ be such  $( N-1) \sqrt{T/2}\le D\le N \sqrt{T/2}$ that is to say
$$ (N-1)\le \sqrt{2}\xi\le N.$$ Then we can find 
$y_{0}=y, y_{1},\dots,y_{N}=z$ with
$$d(y_{i},y_{i+1})\le \sqrt{T/2}.$$
From the chain of inequalities
$$H(s+iT/2,x,y_{i})\le \gamma_{n}H(s+(i+1)T/2,x,y_{i+1}),$$
we obtain
$$H\left(s,x,y\right)\le \gamma_{n}^{N}H(s+N T/2,x,z).$$
Again integrating over $z\in M$ we get
$$H\left(s,x,y\right)\le \frac{\gamma_{n}^{N}}{\vol M}.$$

\endproof

\subsection{Geometric consequence}
\subsubsection{Eigenvalue estimate}
\begin{prop}There are positive constants $c_{n},\alpha_{n}$ such that the first non zero eigenvalue of the Laplacian on $(M,g)$ satisfies
$$\lambda_{1}\ge \frac{c_{n}^{-1-\xi}}{D^{2}}.$$
\end{prop}
\proof Let $f\colon M\rightarrow \R$ be a $L^{2}$ normalized eigenfunction associated to $\lambda_{1}$:
$$\Delta f=\lambda_{1}f\ \mathrm{and} \ \|f\|_{2}=1.$$
Then $(t,x)\mapsto e^{-\lambda_{1}t}f(x)$ is a solution of the heat equation and according to the Bochner formula, the function 
$$u(t,x):=e^{-\lambda_{1}t}|df|(x)$$ satisfies:
$$\frac{\partial u}{\partial t}+\Delta u\le \ricm u.$$
Let $\tau=\min\left\{\frac{T}{2},D^{2}\right\}$.
The function 
$$U(s,x)=\int_{M}H(\tau-s,x,y)u(s,y)\dv_{g}(y)$$ satisfies
\begin{equation*}\begin{split}
\frac{\partial U}{\partial s}(s,x)&=\int_{M}H(\tau-s,x,y)\left(\frac{\partial u}{\partial s}+\Delta u\right)(s,y)\dv_{g}(y)\\
&\le \int_{M}H(\tau-s,x,y)\ricm(y)u(s,y)\dv_{g}(y).
\end{split}\end{equation*}
Hence integrating this inequality, we get:
\begin{equation*}\begin{split}
U(\tau,x)-U(0,x)&=u(\tau,x)-\int_{M}H(\tau,x,y)u(0,y)\dv_{g}(y)\\
&\le  \int_{[0,\tau]\times M}H(\tau-s,x,y)\ricm(y)u(s,y)\dv_{g}(y)ds.
\end{split}\end{equation*}
If we let 
$$L:=\|df\|_{\infty},$$ then using the estimate on the heat kernel \tref{heatnu}, we have
\begin{equation*}\begin{split}
Le^{-\tau \lambda_{1}}&\le \frac{1}{16n}L+\int_{M}H(\tau,x,y)u(0,y)\dv_{g}(y)\\
&\le \frac{1}{16n}L+\frac{c_{n} e^{c_{n}\xi}}{\vol M}\int_{M}u(0,y)\dv_{g}(y)\\
&\le\frac{1}{16n}L+\frac{c_{n} e^{c_{n}\xi}}{\sqrt{\vol M}}\| u\|_{2}\\ 
&\le\frac{1}{16n}L+\frac{c_{n} e^{c_{n}\xi}}{\sqrt{\vol M}}\sqrt{\lambda_{1}}
\end{split}\end{equation*}
We have to distinguish two cases.
\begin{itemize}\item[First case:]  $e^{-\tau \lambda_{1}}\le \frac{1}{8n}$, i.e. $$\lambda_{1}\ge\frac{\log(8n)}{\tau}\ge \frac{\log(8n)}{D^{2}}.$$
\item[Second case:] $e^{-\tau \lambda_{1}}\ge \frac{1}{8n}$,  then we get that
$$L\le 16n \frac{c_{n} e^{c_{n}\xi}}{\sqrt{\vol M}}\sqrt{\lambda_{1}}.$$
As $\int_{M}f(y)\dv_{g}(y)=0$, we can find $o\in M$ such that $f(o)=0$ and then we have for any $x\in M$:
$$|f(x)|^{2}\le |f(x)-f(o)|^{2}\le L^{2}\,D^{2}.$$ Hence
$$1\le L^{2}\,D^{2}\vol(M)\le 256 n^{2}c^{2}_{n} e^{2c_{n}\xi}D^{2}\lambda_{1}.$$
\end{itemize}
\endproof
\subsubsection{Sobolev inequality}
\begin{prop}There is a constant $c_{n}$ such that we have the following Sobolev inequality:
$\forall \psi\in \cC^{\infty}(M)$:
$$\vol^{\frac 2\nu}(M) \|\psi\|^{2}_{\frac{2\nu}{\nu-2}}\le c_{n}^{1+\xi}D^{2}\|d\psi\|_{2}^{2}+\|\psi\|_{2}^{2}.$$
\end{prop}
\proof
We have already shown that if $x\in M$ and $t\in [0,\tau]$ then
$$H(t,x,x)\le \left(\frac{\tau}{t}\right)^{\frac\nu 2} \frac{c_{n} e^{c_{n}\xi}}{\vol M}.$$
So that we can conclude that for 
$$\Gamma:= \left(\frac{\nu}{2 e}\right)^{\nu}\frac{\tau^{\frac\nu 2} c_{n} e^{c_{n}\xi}}{\vol M}$$ then
for all $t>0$ and $x\in M$:
$$H(t,x,x)e^{-\frac{t}{\tau}}\le \Gamma\, t^{-\frac \nu 2}.$$
According to N. Varopoulos (\cite{varo}), there is a constant $c$ that depends only on $n$ such that we have the Sobolev inequality:
$\forall \psi\in \cC^{\infty}(M)$:
$$ \|\psi\|^{2}_{\frac{2\nu}{\nu-2}}\le c\Gamma^{\frac 2\nu}\left( \|d\psi\|_{2}^{2}+\frac{1}{\tau} \|\psi\|_{2}^{2}\right).$$
Using the above eigenvalue estimate, we get for $\psi\in \cC^{\infty}(M)$ such that $\int_{M}\psi(y)\dv_{g}(y)=0$:
$$ \|\psi\|_{2}^{2}\le \lambda_{1}^{-1} \|d\psi\|_{2}^{2}.$$
But
$$c\Gamma^{\frac 2\nu}\frac{1}{\tau}\lambda_{1}^{-1}\le c_{n}\frac{ D^{2} e^{c_{n}\xi}}{\vol^{\frac2\nu} M}$$
hence
$$\max\{c\Gamma^{\frac 2\nu},c\Gamma^{\frac 2\nu}\frac{1}{\tau}\lambda_{1}^{-1}\}\le c_{n}\frac{ D^{2} e^{c_{n}\xi}}{\vol^{\frac2\nu} M}$$ and we finally get that
for all $\psi\in \cC^{\infty}(M)$:
$$\left\|\psi-\frac{1}{\vol M}\int_{M} \psi\,\right\|^{2}_{\frac{2\nu}{\nu-2}}\le c_{n}\frac{ D^{2} e^{c_{n}\xi}}{\vol^{\frac2\nu} M}\|d\psi\|_{2}^{2}.$$
\subsubsection{The doubling condition}
\begin{prop}\label{doubling}There is a constant $c_{n}$ such that if $x\in M$ and $0<r<R\le D$ then
$$\frac{\vol B(x,R)}{R^{\nu}}\le c_{n}^{1+\xi}\frac{\vol B(x,r)}{r^{\nu}}.$$
\end{prop}

Note that this results yields the following global bound on the heat kernel:
\begin{cor}There is a constant $c_{n}$ such that if $x\in M$ and $t>0$ then
$$H(t,x,x)\le \frac{c_{n}^{1+\xi}}{\vol B(x,\sqrt{t})}.$$
\end{cor}

\proof
When $0<s\le t\le \tau$ and $y\in B(x,\sqrt{t})$, we have already shown that:
$$H(s,y,y)\le\left(\frac t s \right)^{\frac\nu2} \frac{\gamma_{n}}{\vol B(x,\sqrt{t})}.$$
Hence when $\Omega\subset B(x,\sqrt{t})$, we have
$$e^{-\lambda_1(\Omega)s}\le \int_{\Omega}H_{\Omega}(s,y,y)dy\le  \int_{\Omega}H(s,y,y)dy\le \left(\frac t s \right)^{\frac\nu2} \frac{\gamma_{n}\vol \Omega}{\vol B(x,\sqrt{t})}.$$
Hence if
$$\vol \Omega\le \frac{1}{2\gamma_{n}}\vol B(x,\sqrt{t})$$ one gets:
$$e^{-\lambda_1(\Omega)t}\le\frac 12$$ and if we choose 
$s=\ln(2)/\lambda_1(\Omega)\le t$, we obtain
$$\frac12\le \gamma_{n}\left(t \, \lambda_1(\Omega)\right)^{\frac 2\nu}\frac{\vol \Omega}{\vol B(x,\sqrt{t})}.$$
Let $0<r\le R\le \sqrt{\tau}$, we  distinguish two cases.
\begin{itemize}\item[First case:] $\vol B(x,r)\ge \frac{1}{2\gamma_{n}}\vol B(x,R)$.

\item[Second case:] $\vol B(x,r)\le \frac{1}{2\gamma_{n}}\vol B(x,R)$, in this case, we have shown that for all
$\Omega\subset B(x,r)$:
$$ \lambda_1(\Omega)\ge \frac{1}{R^{2}} \left(\frac{\vol \Omega}{2\gamma_{n}\,\vol B(x,R)}\right)^{-\frac 2\nu}.$$
According to \cite{carronlambda}, the Sobolev inequality holds:

$\forall \psi\in \cC_0^{\infty}( B(x,r))$:
$$ \|\psi\|^{2}_{\frac{2\nu}{\nu-2}}\le c_{n} \frac{R^{2}}{\left(\vol B(x,R)\right)^{\frac2\nu}}\|d\psi\|_{2}^{2}.$$
And according to the \tref{sobolev} :
$$\vol B(x,r)\ge c_{n}  \frac{\vol B(x,R)}{R^{\nu}}r^{\nu}.$$

\end{itemize}
We have shown that if $\rho=\sqrt{\tau}$, then for all 
$0<r<R\le \rho$:
$$\frac{\vol B(x,r)}{\vol B(x,R)}\ge \min\left\{\frac{1}{2\gamma_{n}}, c_{n}  \frac{r^{\nu}}{R^{\nu}}\right\} \ge \min\left\{c_{n},\frac{1}{2\gamma_{n}}\right\}\frac{r^{\nu}}{R^{\nu}}.$$

In fact this local doubling condition implies a global one: we claim that there is a constant $c_{n}$ such that for every $\theta\ge 1$ and $r\in (0,\rho)$, then
$$\vol B(x,\theta r)\le c_{n}^{1+\theta}\vol B(x, r)$$

It is not clear that this result is new (however look for instance at the hypothesis of \cite[theorem 1.5]{ACDH}).
\begin{lem}\label{locdoublingexp}
Assume that $(X,d,\mu)$ is a measure metric space that satisfies, for some $r_{0}>0$ and $\upgamma>0$, the local doubling condition:
for all $r\in [0,r_{0}]$;
$$\mu(B(x,2r))\le \upgamma \mu(B(x,r))$$
then for every $\theta\ge 1 $ and $r\in [0,r_{0}]$:
$$\mu(B(x,\theta r))\le \upgamma^{50+50 \theta} \mu(B(x,r))$$
\end{lem}
\begin{proof}[Proof of \lref{locdoublingexp}]
By scaling we can assume $r_{0}=1$. Let $R\ge 0$, we have
$$B\left(x,R+\frac{1}{20}\right)=\bigcup_{p\in B(x,R)}B\left(p,\frac{1}{20}\right).$$
Using Vitali's covering lemma, we can find a familly of pairwise disjoint balls
$B(p_{\alpha},1/20)$ such that 
$$B\left(x,R+1/20\right)\subset\bigcup_{\alpha}B\left(p_{\alpha},1/4\right).$$
Hence using the doubling condition:
\begin{equation*}\begin{split}
\mu(\left(B\left(x,R+1/20\right)\right)&\le \sum_{\alpha}\mu(\left(B\left(p_{\alpha},1/4\right)\right)\\
&\le \upgamma^{3}\sum_{\alpha}\mu(\left(B\left(p_{\alpha},1/32\right)\right)
\end{split}
\end{equation*}
But the balls $B\left(p_{\alpha},1/32\right)$ are disjoints and included in $B\left(x,R+\frac{1}{32}\right)$; recall that each $p_{\alpha}\in B\left(x,R\right)$, hence
$$\mu(\left(B\left(x,R+1/20\right)\right) \le \upgamma^{3}\mu(\left(B\left(x,R+1/32\right)\right).$$
So that if $N\in \N\setminus \{0\}$:
\begin{equation*}\begin{split}\mu(\left(B\left(x,N\right)\right)&\le \upgamma^{3}\mu(\left(B\left(x,N-1/10\right)\right)\\
&\le  \upgamma^{30}\mu(\left(B\left(x,N-1\right)\right)\\
&\le \upgamma^{30N-30}\mu(\left(B\left(x,1\right)\right).\end{split}
\end{equation*}
\end{proof}
We can now finish the proof the \pref{doubling}:
$$\frac{\vol M}{D^{\nu}}\le c_{n}^{1+\frac{D}{\rho}}\frac{\rho^{\nu}}{D^{\nu}}
\frac{\vol B(x,\rho)}{\rho^{\nu}}\le  c_{n}^{1+2\xi}\frac{\vol B(x,\rho)}{\rho^{\nu}}.$$

\subsubsection{Poincar\'e inequality}
\begin{prop}There is a constant $c_{n}$ such that for any ball $B$ of radius $r\le \min\{D, (T/2)^{2}\}$ 
satisfies the Poincar\'e inequality:
$$\forall \psi\in \cC^{\infty}(B)\, ,\ 
\left\| \psi-\int_{B}\psi\frac{\dv_{g}}{\vol B}\right\|_{2}\le c_{n}r\left\| d\psi\right\|_{2}.$$
And for any  any ball $B$ of radius $r$:
$$\forall \psi\in \cC^{\infty}(B)\, ,\ 
\left\| \psi-\int_{B}\psi\frac{\dv_{g}}{\vol B}\right\|_{2}\le c_{n}^{1+\xi}r\left\| d\psi\right\|_{2}.$$

\end{prop}
\proof According to the results of L. Saloff-Coste and A. Grigor'yan \cite{G91, Saloff-Coste,Saloff-Coste1}, we only need to show that if $t\le \tau$ and all $y\in B(x,\sqrt{t})$:
$$\frac{\epsilon_{n}}{\vol B(x,\sqrt{t})}\le H(t,x,y).$$
But we already know, that if $t\le \tau$ and  $y\in B(x,\sqrt{t})$ then $$ c^{-1}_{n} H(t,x,y)\le H(t,x,x)\le c_{n} H(t,x,y).$$
Hence for all $\delta\in (0,1):$
\begin{equation*}\begin{split}c_{n}\vol B(x,\sqrt{t}) H(t,x,y)& \ge  \vol B(x,\sqrt{t}) H(t,x,x)\\
& \ge \int_{ B(x,\sqrt{t})}H(t,x,z)\dv_{g}(z)\\
& \ge \delta^{\nu}\int_{ B(x,\sqrt{t})}H(\delta t,x,z)\dv_{g}(z)\\
&= \delta^{\nu}\left(1-\int_{M\setminus B(x,\sqrt{t})}H(\delta t,x,z)\dv_{g}(z)\right).\end{split}
\end{equation*}
Our Harnack type estimate yields that 
$$\int_{M\setminus B(x,\sqrt{t})}H(\delta t,x,z)\dv_{g}(z)\le\int_{B(x,\sqrt{\delta t})\times \left(M\setminus B(x,\sqrt{t})\right)}H(\delta t,\xi,z) \frac{\dv_{g}(\xi)\dv_{g}(z)}{\vol B(x,\sqrt{\delta t})}.$$
Moreover \begin{align*}
\int_{B(x,\sqrt{\delta t})\times \left(M\setminus B(x,\sqrt{t})\right)}&H(\delta t,\xi,z)\dv_{g}(\xi)\dv_{g}(z)\\
&=\sum_{k=1}^{\infty }\int_{B(x,\sqrt{\delta t})\times\left( B(x,(k+1)\sqrt{t})\setminus B(x,k\sqrt{t})\right)}H(\delta t,\xi,z)\dv_{g}(\xi)\dv_{g}(z).\end{align*}
The Davies-Gaffney estimate yields that if $k^{2}>\delta$, then 
\begin{equation*}\begin{split}
\int_{B(x,\sqrt{\delta t})\times \left(B(x,(k+1)\sqrt{t})\setminus B(x,k\sqrt{t})\right)}&H(\delta t,\xi,z)\dv_{g}(\xi)\dv_{g}(z)\le\\
 e^{-\frac{(k-\sqrt{\delta})^{2}}{\delta}}&\left(\vol B(x,\sqrt{\delta t})\times \vol  B(x,(k+1)\sqrt{t})\right)^{\frac12}.\end{split}
\end{equation*}

But the \lref{locdoublingexp} implies that 
$$\vol  B(x,(k+1)\sqrt{t})\le e^{c_{n}+c_{n} \frac{k+1}{\sqrt{\delta}}}  \vol B(x,\sqrt{\delta t}).$$
And eventually, one get
$$\int_{B(x,\sqrt{\delta t})\times \left(M\setminus B(x,\sqrt{t})\right)}H(\delta t,\xi,z)\dv_{g}(\xi)\dv_{g}(z)\le \sum_{k=1}^{\infty }e^{c_{n}+c_{n} \frac{k+1}{\sqrt{\delta}}-\frac{(k-\sqrt{\delta})^{2}}{\delta}}.$$
We can choose $\delta=\delta_n$ to be small enough to that this sum is less that $1/2$, and then we get
$$c_{n}\vol B(x,\sqrt{t}) H(t,x,y) \ge\frac 12 \delta_{n}^{\nu}.$$
 \endproof
\subsubsection{Betti number}
\begin{prop}There is a constant $c_n$ such that 
$b_{1}(M)\le n+\frac14+\xi c_n^\xi.$ Moreover there is a constant $\upepsilon_{n}>0$ such that 
if $\xi(M,g)< \upepsilon_{n}$, then $b_{1}(M)\le n$.\end{prop}
\proof This result relies on an improvement of the upper bound on the heat kernel.
We have shown that 
$$\left\| e^{-\frac T2 \Delta}\right\|_{L^{2}\to L^{\infty}}\le \frac{c_{n}^{1+\xi}}{\sqrt{\vol M}}.$$
Hence if $P$ is the $L^{2}$-projection on the vector space of constant function  then
\begin{equation*}\begin{split}
\left\| e^{- T \Delta}-P\right\|_{L^{2}\to L^{\infty}}&=\left\|e^{-\frac T2 \Delta}\left( e^{-\frac T2 \Delta}-P\right)\right\|_{L^{2}\to L^{\infty}}\\
&\le \left\|e^{-\frac T2 \Delta}\right\|_{L^{2}\to L^{\infty}}\left\|\left( e^{-\frac T2 \Delta}-P\right)\right\|_{L^{2}\to L^{2}}\\
&\le \frac{c_{n}^{1+\xi}}{\sqrt{\vol M}} \lambda_{1}^{-\frac12}\left\|d\left( e^{-\frac T2 \Delta}\right)\right\|_{L^{2}\to L^{2}} \\
&\le \frac{c_{n}^{1+\xi}}{\sqrt{\vol M}} \lambda_{1}^{-\frac12}\frac{1}{\sqrt{T}}\end{split}
\end{equation*}
Recall our lower bound on $\lambda_{1}$ and the fact that 
$T\xi^{2}=D^{2},$ and we get
$$
\left\| e^{- T \Delta}\right\|_{L^{2}\to L^{\infty}}\le \left\| e^{- T \Delta}-P\right\|_{L^{2}\to L^{\infty}} +\left\| P\right\|_{L^{2}\to L^{\infty}}\le \frac{1+\xi c_{n}^{1+\xi}}{\sqrt{\vol M}}.$$

If $\alpha\in \cC^{\infty}(T^{*}M)$ satisfies $d\alpha=d^{*}\alpha=0$, then the Bochner formula implies that 
$$\Delta |\alpha|\le \ricm |\alpha|.$$
Hence
\begin{equation*}\begin{split}
|\alpha|(x)&\le \left(e^{-T\Delta}|\alpha|\right)(x)+\int_{0}^{T}\left(e^{-s\Delta}\ricm |\alpha|\right)(x)ds\\
&\le \frac{1+\xi c_{n}^{1+\xi}}{\sqrt{\vol M}}\|\alpha\|_{2}+\ka_{T}(\ricm)\|\alpha\|_{\infty}\\
&\le\frac{1+\xi c_{n}^{1+\xi}}{\sqrt{\vol M}}\|\alpha\|_{2}+\frac{1}{16n}\|\alpha\|_{\infty}\end{split}
\end{equation*}
Finally we obtain that for any $\alpha\in \cH^{1}(M,g)=\{\alpha\in \cC^{\infty}(T^{*}M)\colon\ d\alpha=d^{*}\alpha=0\}$:
$$\|\alpha\|_{\infty}\le\frac{1}{1-\frac{1}{16n}} \frac{1+\xi c_{n}^{1+\xi}}{\sqrt{\vol M}}\|\alpha\|_{2}.$$
The Grothendieck theorem \cite[Theorem 5.1]{Ru} (see also \cite[Th\'eor\`eme 4]{GaMe}) yields that 
$$b_{1}(M)=\dim \cH^{1}(M,g)\le n\left(\frac{1+\xi c_{n}^{1+\xi}}{1-\frac{1}{16n}}\right)^{2}.$$
Then a bit of arithmetic implies the proposition. \endproof
\subsection{Euclidean type estimate}
\subsubsection{Improvement}
We assume now that 
\begin{equation}\label{Katoeucli}
\ka_{T}\le \frac{1}{16n}\ \mathrm{and}\ \int_{0}^{T}\frac{\sqrt{\ka_{s}}}{s}\,ds\le \Lambda.
\end{equation}
According to the \pref{LYI}, that if $u\colon [0,T]\times M\rightarrow\R_{+}$ is a positive solution of the heat equation then
$$-\frac{1}{u}\frac{\partial u}{\partial t}\le \frac{n}{2t}+C_{n }\frac{\sqrt{\ka_{t}}}{t}.$$
Hence if $0<s<t\le T$ then
$$u(s,x)\le \left(\frac t s\right)^{\frac n 2}e^{\Lambda C_{n}}u(t,x).$$
So that the heat kernel satisfies if $0<s<t\le T$ and $x\in M$ then
$$s^{\frac n 2}H(s,x,x)\le e^{\Lambda C_{n}} t^{\frac n 2}H(t,x,x).$$
And looking at the behaviour when $s\to 0+$, we get
$$ \frac{e^{-\Lambda C_{n}}}{t^{\frac n 2}}\le H(t,x,x).$$
Using the upper bound of \ref{heatnu}, we get for $0<t\le \tau$ and $x\in M$:
$$H(s,x,x)\le  \frac{e^{\Lambda C_{n}}}{t^{\frac n 2}} \frac{c_{n} e^{c_{n}\xi}}{\vol M}$$ and
$$\vol B\left(x,\sqrt{t}\right)\le c_{n}e^{\Lambda C_{n}}t^{\frac n 2} $$
As a consequence we can improve our earlier results:
\begin{prop}If $(M^{n},g)$ is a compact Riemannian manifold, such that for some $T>0$ the parabolic Kato constant of $\ricci_{-}$ satisfies (\ref{Katoeucli}), then
\begin{enumerate}[--]
\item The Sobolev inequality holds:
$\forall \psi\in \cC^{\infty}(M)$:
$$\vol^{\frac 2 n}(M) \|\psi\|^{2}_{\frac{2n}{n-2}}\le c_{n}^{1+\xi+\Lambda}D^{2}\|d\psi\|_{2}^{2}+\|\psi\|_{2}^{2}.$$
\item There is a constant $c_{n}$ such that if $x\in M$ and $0<r\le D$ then
$$ c_{n}^{1+\xi+\Lambda } \frac{\vol M}{D^{n}}\le \frac{\vol B(x,r)}{r^{n}}\le c_{n}^{2+\xi+\Lambda }.$$
\end{enumerate}
\end{prop}

\subsubsection{$L^p$-Kato class}
\begin{defi}If $p\ge 1$ and $T>0$, we defined the $L^p$-Kato constant of $\ricm$ by
$$\ka_{p,T}(\ricm)^{p}=(\diam M)^{2p-2}\sup_{x\in M}\int_{0}^{T}H(s,x,y)\ricm^{p}(y)\dv_{g}ds.$$\end{defi}
Let $q=p/(p-1)$,
then with the H\"older inequality, we get for $0\le \underline{T}\le T$ and $x\in M$:
\begin{equation*}\begin{split}
\int_{0}^{ \underline{T}}H(s,x,y)\ricm(y)\dv_{g}ds&\le  \underline{T}^{\frac1q} \left(\int_{0}^{\underline{T}}H(s,x,y)\ricm^{p}\dv_{g}ds\right)^{\frac{1}{p}}\\
&\le\ka_{p,T}(\ricm)\left(\frac{ \underline{T}}{D^{2}}\right)^{\frac{1}{q}} .
\end{split}
\end{equation*}
Hence for 
$$\underline{T}=\min\left\{ T, \left(16n\ka_{p,T}(\ricm)\right)^{-q}D^{2}\right\},$$
one gets 
$$\ka_{\underline{T}}(\ricm)\le \frac{1}{16n}\ \mathrm{and}\ \int_{0}^{\underline{T}}\frac{\sqrt{\ka_{s}(\ricm)}}{s}ds\le q/(2\sqrt{n})$$
Hence one gets the following
\begin{thm}\label{thm:Lp} Let $(M^{n},g)$ be a compact Riemannian manifold of dimension $n$ and let $p>1$ and $q=p/(p-1).$
Define
$$\xi=\max\left\{ \frac{D}{\sqrt{T}}, \left(16n\ka_{p,T}(\ricm)\right)^{q/2}\right\}$$ then there is a constant $\upgamma$ that depends only of $n,p$ such that the following properties holds:
\begin{enumerate}[-]	
\item The first non zero eigenvalue of the Laplacian satisfies $$ \lambda_{1}\ge \frac{\upgamma^{-1-\xi}}{D^{2}}.$$
\item $b_{1}(M)\le \upgamma^{1+\xi}$
\item For any $0<r\le R \le D$:
$$ \ \frac{\vol B(x,R)}{R^{n}}\le \upgamma^{1+\xi} \frac{\vol B(x,r)}{r^{n}}\le \upgamma^{2+\xi}.   $$
\item One gets the Euclidean Sobolev inequality:
$\forall \psi\in \cC^{\infty}(M)$:
$$\vol^{\frac 2 n}(M) \|\psi\|^{2}_{\frac{2n}{n-2}}\le \upgamma^{1+\xi}D^{2}\|d\psi\|_{2}^{2}+\|\psi\|_{2}^{2}.$$\end{enumerate}
\end{thm}
\subsection{$\Qc$-curvature and bound on the Kato constant}
It turns out that the $\Qc-$curvature gives a natural control on the $L^2$ Kato constant of the Ricci curvature. Recall that if $(M,g)$ is Riemannian manifold of dimension $n\ge 4$, its $\Qc-$curvature is defined by:
$$\Qc_{g}=\frac{1}{2(n-1)}\Delta \scal_{g}-\frac{2}{(n-2)^{2}}|\ricci|^{2}+c_{n}\scal_{g}^{2},$$
where $c_{n}=\frac{n^{3}-4n^2+16n-16}{8n(n-1)^{2}(n-2)^{2}}$.
Recently, M. Gursky and A. Machioldi (\cite{GM}) have discovered some new maximum principle for the Paneitz operator (which describes the conformal change of the $\Qc$-curvature) when the $\Qc$-curvature is non negative and the scalar curvature is positive. It turns out that these hypothesis yields a bound on the $L^2$ Kato constant of $\ricm$.
\begin{prop} Let $(M^{n},g)$ be a compact Riemannian manifold of dimension $n\ge 4$ such that:
$$0\le\Qc_{g} \ \mathrm{and}\ 0\le \scal_{g}\le \kappa^{2}/D^{2}.$$
Then the conclusion of the \tref{thm:Lp} are satisfied for $\xi=\upepsilon_n\, \kappa$.
\end{prop}
Indeed if $T>0$, the hypothesis $\Qc_{g}\ge 0$ furnishes:
\begin{equation*}
\begin{split}\frac{2}{(n-2)^{2}} \int_{[0,T]\times M}&H(s,x,y)|\ricci|^{2}(y)ds\dv_{g}(y)\\
&\le \frac{1}{2(n-1)}\left(\scal_{g}(x)-\left(e^{T\Delta}\scal_{g}\right)(x)\right)+c_{n}
\frac{\kappa^{4}T}{D^{4}}\\
&\le \frac{1}{2(n-1)}\frac{\kappa^2}{D^{2}}+c_{n}
\frac{\kappa^{4}T}{D^{4}}.
\end{split}
\end{equation*}
If we choose $T=\epsilon_{n}D^{2}/\kappa^{2}$, one get
$\ka_{2,T}(\ricm)^{2}\le \alpha_{n}\kappa^2$ and
$$\ka_{T}(\ricm)^{2}\le  \alpha_{n}\epsilon_{n}.$$
\subsection{Localisation} It has been noticed by Qi S. Zhang and M. Zhu, that the results of this section can be localized on a geodesic balls provided one get a good cut-off function:
\begin{prop}\label{ZZlocal}
Let $(M^{n},g)$ be a compact Riemannian manifold of dimension $n$ and consider $B(x,R)\subset \Omega$ be a geodesic ball included in a relatively compact open subset $\Omega$. Assume that 
there is a smooth function $\xi$ with compact support in $\Omega$ such that
$\xi=1$ on $B(x,R)$ and $$|d\xi|^{2}+|\Delta \xi|\le c/R^{2}.$$
Let $H_{\Omega}$ be the heat kernel on $\Omega$ for the Dirichlet boundary condition and  consider the assumptions 
\begin{enumerate}[A)]
\item For all $x\in \Omega$:
$$\int_{[0,\eta R^{2}]\times\Omega} H_{\Omega}(s,x,y)\ricm(y)\dv_{g}(y)\le \frac{1}{16 n}.$$
\item For some $p>1$, $\Lambda\in \R_+$ and for all $x\in \Omega$: $$\int_{[0,\eta R^{2}]\times\Omega} H_{\Omega}(s,x,y)\ricm^p(y)\dv_{g}(y)\le \Lambda R^{2p-1}.$$
\end{enumerate}
Under the condition A) or B), there is a constant $\upgamma$ that depends only on $n,c,\eta$ or on $n,c,p,\Lambda$ such that 
\begin{enumerate}[i)]
\item For any $x\in B(p,R/2)$ and $0<t\le R^{2}$ then 
$$H(t,x,x)\le \frac{\upgamma}{\vol B(x,\sqrt{t})}.$$
\item For any $x\in B(p,R/2)$ and $0\le r\le R/2$:
$$\vol B(x,r)\le \upgamma \vol B(x,2r)$$
\item  For any $x\in B(p,R/2)$ and $0\le r\le R/2$, the ball $B(x,r)$ satisfies the Poincar\'e inequality for a constant $\upgamma r$.\end{enumerate}
And if the condition B) is satisfied then moreover:
 For any  $x\in B(p,R/2)$ and $0\le s\le r\le R/2$: 
 $$\frac{\vol B(x,r)}{r^n}\le \upgamma \frac{\vol B(x,s)}{s^n}\le \upgamma^2 \omega_n.$$

\end{prop}

\subsection{Appendix : comparison between different Kato constant}
In this section, we compare the parabolic Kato constant and the elliptic Kato constant:
If $(M,g)$ is a Riemannian manifold and $\Delta$ is the Friedrichs extension of the Laplacian. If $V\in L^1_{\loc}$ is a non negative function, then we define the elliptic Kato constant 
\begin{equation}
\Ka_{\lambda}(V)=\sup_{x\in M}\int_{M}G_{\lambda}(x,y)V(y)\dv_{g}(y)=\left\|\left(\Delta+\lambda\right)^{-1}V\right\|_{\infty}
\end{equation}
 and the parabolic Kato constant of $V$ by
\begin{equation}
\ka_{T}(V)=\sup_{x\in M}\int_{0}^{T}\int_{M}H(t,x,y)V(y)\dv_{g}(y)dt=\left\| \int_{0}^{T}e^{-t\Delta}Vdt\right\|_{\infty}
\end{equation}
\begin{lem}\label{equiKato}
$$e^{-1}\Ka_{\frac 1T}\le \ka_{T}(V)\le \frac{e}{e-1} \Ka_{\frac 1T}.$$
\end{lem}
\proof 
We have the relationship:
$$\left(\Delta+\lambda\right)^{-1}V=\int_{0}^{+\infty}e^{-\lambda t}e^{-t\Delta}Vdt.$$
Hence
$$\left(\Delta+\lambda\right)^{-1}V=\int_{0}^{T}e^{-\lambda T}e^{-t\Delta}Vdt.$$
Hence the lower bound.
We have also
$$\left(\Delta+\lambda\right)^{-1}V=\sum_{k=0}^{\infty}\int_{kT}^{(k+1)T}e^{-\lambda t}e^{-t\Delta}Vdt.$$
The heat semi-group is sub-Markovian, hence $t\mapsto \|e^{-t\Delta}V\|_{\infty}$ is non increasing and
\begin{equation*}\begin{split}
 \left\|\left(\Delta+\lambda\right)^{-1}V\right\|_{\infty}&\le \sum_{k=0}^{\infty}\int_{0}^{T}e^{-\lambda k T}\left\|e^{-t\Delta}Vdt\right\|_{\infty}\\
 &\le \sum_{k=0}^{\infty}\int_{0}^{T}e^{-\lambda k T} \Ka_{T}(V)=\frac{1}{1-e^{-\lambda T}}\Ka_{T}(V)
\end{split}
\end{equation*}

\endproof

\section{Volume growth estimate : global results}
\subsection{The setting} We consider a complete Riemannian manifold $(M^{n},g)$ satisfying the Euclidean Sobolev inequality:
$\forall \psi\in \cC_0^{\infty}(M)$:
$$\mu\, \| \psi\|_{\frac{2n}{n-2}}^{2}\le \|d\psi\|_{2}^{2}.$$
We also assume that for some $\delta>0$, the \sch operator $\Delta-(1+\delta)(n-2)\ricm$ is gaugeable: there is $h\colon M\rightarrow [1,\gamma]$ such that 
$$\left\{\begin{array}{l}
\Delta h-(1+\delta)(n-2)\ricm h=0\\
1\le h\le \gamma
\end{array}
\right.$$
According to the remark (\ref{remGaugeable} b), we can assume that 
$$0<\delta<\frac{n-2}{n(3n-4)}$$ so that 
$$2\le 2\frac{n-2}{n-1}\frac{1}{1+\delta}\ \mathrm{and}\ 2(n-1)\frac{\delta}{1+\delta}<(n-2)\sqrt{\,\frac{\delta}{1+\delta}\,}.$$
We fix now $o\in M$ and consider the Green kernel with pole at $o$:
$$G(o,x)=\frac{1}{b(x)^{n-2}},$$
where we choose to normalize the Green kernels so that 
$$\Delta_{x} G(o,x)=(n-2)\mathrm{vol} \,\bS^{n-1}\delta_{o}$$ hence
near $o$ we have
$$b(x)\simeq d(o,x).$$
For $p\le (1+\delta)(n-2)$, we let $G_{p}$ be the Green kernel of the \sch operator $\Delta-p\ricm$.
The \tref{global} will be a direct consequence of the following:
\begin{prop}\label{finalgreen}
$$\frac{|db|^{(1+\delta)(n-2)}(x)}{b^{n-2}(x)}\le G_{(1+\delta)(n-2)}(o,x).$$
\end{prop}
\begin{proof}[End of the proof of \tref{global} while assuming \pref{finalgreen}]
The gaugeability of the \sch operator $\Delta -(1+\delta)(n-2)\ricm$ and the Sobolev inequality implies that 
$$G_{(1+\delta)(n-2)}(o,x)\le \frac{c_{n}\gamma^{n}}{\mu^{\frac n 2}}\frac{1}{d(o,x)^{n-2}}.$$
Let $r(x):=d(o,x)$ so that we have
$$\frac{|db|(x)}{b^{\frac{1}{1+\delta}}(x)}\le \left(\frac{c_{n}\gamma^{n}}{\mu^{\frac n 2}}\right)^{\frac{1}{(1+\delta)(n-2)}} \frac{1}{r^{\frac{1}{1+\delta}}(x)}.$$
Hence integrating along a minimizing geodesic joining $o$ and $x$, we get
$$b(x)^{\frac{\delta}{1+\delta}}\le \left(\frac{c_{n}\gamma^{n}}{\mu^{\frac n 2}}\right)^{\frac{1}{(1+\delta)(n-2)}}r(x)^{\frac{\delta}{1+\delta}}$$ and
$$b(x)\le B r(x)$$ 
where $$B=\left(\frac{c_{n}\gamma^{n}}{\mu^{\frac n 2}}\right)^{\frac{1}{\delta(n-2)}}.$$
Hence the geodesic ball $B(o,R)$ is include in the sublevel set $\{b\le B R\}$ and from \tref{sobolev}ii-b), we have
$$\vol B(o,R)\le B^{n}\mu^{-\frac{n}{n-2}}\, R^{n}.$$
\end{proof}

In order to prove the above proposition, we will use the following
\begin{lem}\label{compgreen}Let $\frac{n-2}{n-1}\le p\le (1+\delta)(n-2)$ and $\alpha\ge 2$. If
$$\int_{M\setminus B(o,1)}\left( \frac{|db|^{p}}{b^{n-2}}\right)^{\alpha}\dv_{g}<\infty$$ then
$$\frac{|db|^{p}(x)}{b^{n-2}(x)}\le G_{p}(o,x).$$
\end{lem}
\begin{proof}[Proof of \lref{compgreen}]
Indeed by (\ref{yau2}), we know that
$$(\Delta-p\ricm) \frac{|db|^{p}}{b^{n-2}}\le 0\hspace{0,5cm}\mathrm{on}\  M\setminus\{o\}.$$
Hence the gaugeability of the operator $\Delta-p\ricm$ implies that  if $x\in M\setminus\{o\}$ then
$$\frac{|db|^{p}(x)}{b^{n-2}(x)}\le \frac{C}{d(o,x)^{\frac{n}{\alpha}}}\left(\int_{ B(x,d(o,x)/2)}\left( \frac{|db|^{p}}{b^{n-2}}\right)^{\alpha}\dv_{g}\right)^{\frac1\alpha}.$$
Hence 
$$\lim_{\infty }\frac{|db|^{p}}{b^{n-2}}=0$$ 
According to what we said in the subsection \ref{nearpole}, there is some $\tau(\epsilon)$ such that 
$\lim_{\epsilon\to 0}\tau(\epsilon)=0$ and
$$\frac{|db|^{p}(x)}{b^{n-2}(x)}\le (1+\tau(\epsilon)) G_{p}(o,x)\hspace{0,5cm}\mathrm{on}\ \partial B(o,\epsilon).$$
The Maximum principle implies then 
$$\frac{|db|^{p}(x)}{b^{n-2}(x)}\le (1+\tau(\epsilon)) G_{p}(o,x)+\sup_{z\in \partial B(o,R)}\frac{|db|^{p}(z)}{b^{n-2}(z)}\hspace{0,5cm}\mathrm{on}\ B(o,R)\setminus B(o,\epsilon).$$
Letting $\epsilon\to 0$ and $R\to\infty$, we get the result.
\end{proof}
\subsection{Bound on the log derivative of the Green kernel}
Let $p_{0}=\frac{n-2}{n-1}.$ We have already noticed that 
$$\int_{M\setminus B(o,1)}|d_{x}G(o,x)|^{2}\dv_{g}<\infty$$ hence
$$\int_{M\setminus B(o,1)}\left( \frac{|db|^{p_{0}}}{b^{n-2}}\right)^{2\frac{n-1}{n-2}}\dv_{g}<\infty$$
hence by the \lref{compgreen}, we have
$$\frac{|db|^{p_{0}}(x)}{b^{n-2}(x)}\le G_{p_{0}}(o,x).$$

Our main tool is the universal Hardy inequality (\ref{Hardy}):
$\forall  \psi\in \cC_0^{\infty}(M)$:
$$\frac{(n-2)^{2}}{4}\int_{M} \frac{|db|^{2}}{b^{2}}\psi^{2}\dv_{g}\le \int_{M}|d\psi|^{2}\dv_{g}.$$
The gaugeability of the \sch operator $\Delta-(1+\delta)(n-2)\ricm$ implies that this operator is non negative and we have the following Hardy type inequality:
$\forall  \psi\in \cC_0^{\infty}(M)$:
$$\frac{(n-2)^{2}}{4}\frac{\delta}{1+\delta}\int_{M} \frac{|db|^{2}}{b^{2}}\psi^{2}\dv_{g}\le \int_{M}\left[|d\psi|^{2}-(n-2)\ricm \psi^2\right]\dv_{g}.$$
When $p\in [p_{0},n-2]$, using the function $\psi=\xi \frac{|db|^{p}}{b^{n-2}}$ where $\xi$ is a Lipschitz function with compact support in $M\setminus \{o\}$ one get
\begin{equation}\label{Hardynew}
\frac{(n-2)^{2}}{4}\frac{\delta}{1+\delta}\int_{M} \frac{|db|^{2+2p}}{b^{2(n-1)}}\xi^{2}\dv_{g}\le \int_{M}|d\xi|^{2}\frac{|db|^{2p}}{b^{2(n-2)}}\dv_{g}.
\end{equation}

Assume that for some $p\in [p_{0},n-2]$, we have
$$\frac{|db|^{p}(x)}{b^{n-2}(x)}\le G_{p}(o,x).$$
We know that the measure $G_{p}(o,x)^{2}\dv_{g}(x)$ is parabolic on $M\setminus B(o,1)$ (see \ref{Greenpara}), so that the inequality (\ref{Hardynew}) is valid for $\xi$ a Lipschitz function that is zero in $B(o,1/2)$ and is equal to $1$ outside $B(o,1)$. In particular, one gets
$$\int_{M\setminus B(o,1)} \frac{|db|^{2+2p}}{b^{2(n-1)}}\dv_{g}<\infty$$
If $\bar p=(1+p)\frac{n-2}{n-1}$ one get
$$\int_{M\setminus B(o,1)} \left(\frac{|db|^{\bar p}}{b^{n-2}}\right)^{2\frac{n-1}{n-2}}\dv_{g}<\infty$$
and one gets
$$\frac{|db|^{\bar p}(x)}{b^{n-2}(x)}\le G_{\bar p}(o,x).$$
If we let $p_{k}=(n-2)-\left(\frac{n-2}{n-1}\right)^{k}\left(n-2-p_{0}\right)=\left(1+p_{k-1}\right)\frac{n-2}{n-1},$ our argumentation yields that for all $k\in \N$:
$$\frac{|db|^{p_{k}}(x)}{b^{n-2}(x)}\le G_{ p_{k}}(o,x)\le G_{ n-2}(o,x).$$
Hence letting $k\to\infty$, we obtain the following estimate on the log derivative of the Green kernel:
$$\frac{|db|^{n-2}(x)}{b^{n-2}(x)}\le G_{ n-2}(o,x).$$
\subsubsection{Improvement}
We will use again the inequality (\ref{Hardynew}), for $p=(n-2)$. This inequality is still valid for any  $\xi$ a Lipschitz function that is zero in near $o$ and is equal to $1$ outside a compact set, we will use it with
$$\xi(x)=f(b)$$ where $f\colon (0,\infty)\rightarrow \R$ is a smooth function that is $1$ outside a compact set and $0$ near $o$:
$$\frac{(n-2)^{2}}{4}\frac{\delta}{1+\delta}\int_{M} \frac{|db|^{2(n-1)}}{b^{2(n-1)}} f^{2}(b)\dv_{g}\le \int_{M}\frac{|db|^{2(n-1)}}{b^{2(n-2)}}(f'(b))^{2}\dv_{g}.$$
We introduce the measure $\m$ on $[1,\infty]$ defined by
$$\m([1,R])= \int_{1\le b\le R}\frac{|db|^{2(n-1)}}{b^{2(n-2)}}\dv_{g}.$$
and we get the Hardy type inequality:
$\forall f\in C^{\infty}_{0}((1,\infty))$:
$$\frac{(n-2)^{2}}{4}\frac{\delta}{1+\delta}\int_{1}^{\infty} \frac{1}{t^{2}} f^{2}(t)\dm(t)\le \int_{1}^{\infty} (f'(t))^{2}\dm(t).$$
Using now the \pref{hardyvol}, we get:
$$\int_{R}^{\infty} \frac{1}{t^{2}}\dm(t)=\int_{R\le b }\frac{|db|^{2n-2}}{b^{2(n-1)}}\dv_{g}\le \frac{C}{R^{(n-2)\sqrt{\,\frac{\delta}{1+\delta}\,}}}.$$
Hence for all $m<(n-2)\sqrt{\,\frac{\delta}{1+\delta}\,}$, with $p=\frac{n-2}{1-\frac m2\frac{n-2}{n-1}}$, one get
$$\int_{M\setminus B(o,1)}\left( \frac{|db|^{p}}{b^{n-2}} \right)^{2\frac{n-1}{n-2}-\frac{m}{n-2}}<\infty$$
Our restriction of $\delta$ makes possible the choice
$m=(n-1)\frac{\delta}{1+\delta}$ and $p=(n-2)(1+\delta).$
Then according to the \lref{compgreen}, we have proven the \pref{finalgreen}.
\subsection{Others properties}
In fact, once one has proven the Euclidean volume growth, the other properties of \tref{globalsuite} follow. Indeed, $(M,g)$ is doubling and sastifies the upper (\ref{LY}) bound:
$$\forall t>0,x,y\in M\colon\  H(t,x,y)\le \frac{c e^{-\frac{d^2(x,y)}{5t}}}{\vol B(x,\sqrt{t})}.$$
Moreover, according to the (\rref{remGaugeable}-b), the operator $L:=\Delta-\ricm$ is jaugeable and its heat kernel satisfies the same  upper (\ref{LY}) bound. Let $\vec H$ be the heat kernel associated with the Hodge-deRham Laplacian $\vec \Delta$ acting on $1$-forms:
$$\vec \Delta=\nabla^*\nabla+\ricci.$$
By domination, we know that 
$$\forall t>0,x,y\in M\colon\  \left|\vec H(t,x,y)\right| \le H_L(t,x,y)\le \frac{c e^{-\frac{d^2(x,y)}{5t}}}{\vol B(x,\sqrt{t})}.$$
The results and the proof of \cite[theorem 5.5]{CDfull} implies that 
\begin{enumerate}[-]
\item the heat kernel of $(M,g)$ satisfies the \ref{LY} estimates, hence $(M,g)$ also satisfies the Poinca\'e inequalities. 
\item The Riesz transform $d\Delta^{-\frac 12}$ is $L^p\to L^p$ bounded for every $p\ge 2$.
\end{enumerate}
And according to \cite{CD_TAMS}, the Riesz transform  is also $L^p\to L^p$ bounded for every $p\in (1,2].$

\subsection{Proof of the volume estimate in  \tref{globalbis}}
We assume now that $(M^{n},g)$ is a complete Riemannian manifold that satisfies the Euclidean Sobolev inequality and that there is a $\delta>0$ \footnote{We can always assume $\delta<(n-2)/(n(3n-4)).$} such that the \sch operator $$\Delta-(n-2)(1+\delta)\ricm$$ is gaugeable at infinity: there is compact $K\subset M$ and a function $h\colon M\setminus K\rightarrow [1,\gamma]$ such that 
$$\Delta h-(n-2)(1+\delta)\ricm h=0.$$
We can find $W\in C^\infty_{0}(M)$ non negative such that the  \sch operator $$L:=\Delta+W-(n-2)(1+\delta)\ricm$$
is gaugeable.\footnote{Indeed if $\bar h\colon M\rightarrow [1/2,2\gamma]$ is an extension on $h$, then there is a bounded function $q$  with compact support such that $$\Delta\bar  h+q\bar  h-(n-2)(1+\delta)\ricm\bar  h=0.$$
Hence the \sch operator $P:=\Delta+q-(n-2)(1+\delta)\ricm$ is gaugeable and by \cite{DevKato}[Theorem 3.2], for any non negative function $\cV$ with compact support the operator $P+\cV$ is gaugeable.}

We will note $G_{L}$ the Green kernel of the operator $L$.
Let $o\in M$ , we still define
$$G(o,x)=\frac{1}{b_{o}^{n-2}(x)}.$$

Our previous argument can be used to show the following: 
Let  $\rho>0$ such that $\supp W\subset B(o,\rho)$ and $(n-2)/(n-1)\le p\le (n-2)(1+\delta)$. 

\begin{lem}\label{greenesti} If 
$$\lim_{\infty}\frac{|db_{o}|^{p}}{b_{o}^{n-2}}=0$$ then
\begin{enumerate}[i)]
\item on $M\setminus B(o,2\rho)$: 
$$\frac{|db_{o}|^{p}(x)}{b_{o}^{n-2}(x)}\le \frac{A}{a}\,G_{L}(o,x)$$
where $A=\sup_{x\in \partial B(o,2\rho)}\frac{|db_{o}(x)|^{p}}{b_{o}^{n-2}(x)}$ and 
$a=\inf_{x\in \partial B(o,2\rho)} G_{L}(o,x)$.
\item If $x\in M\setminus\{o\}$ then $$\frac{|db_{o}|^{p}(x)}{b_{o}^{n-2}(x)}\le G_{L}(o,x)+\int_{\supp W} G_{L}(x,y)W(y) \frac{|db_{o}|^{p}(y)}{b_{o}^{n-2}(y)}\dv_{g}(y).$$
\end{enumerate}
\end{lem}
The same argumentation yields that the hypothesis of the lemma is satisfied for for $p=(n-2)(1+\delta)$:
\begin{prop}\label{volo}Assume that $(M,g)$ is a complete Riemannian manifold that satifies the Euclidean Sobolev inequality and assume that for some $\delta>0$, the \sch operator $\Delta-(n-2)(1+\delta)\ricm$ is gaugeable at infinity. Let $o\in M$. There are positive constants $c,\epsilon$ that depends on $(M,g)$ and $o$ such that 
\begin{enumerate}
\item For all $R>0$: $\vol B(o,R)\le c R^{n}.$
\item For any $x\in M$, the Green kernel satisfies
\begin{equation*}\label{dbgreen}
\left(\frac{\epsilon}{ d(o,x)}\right)^{n-2}\le G(o,x)\le \frac{1}{\left(\epsilon d(o,x)\right)^{n-2}}
\end{equation*}
\item If $b$ is defined by $$G(o,x)=b(x)^{2-n}$$ then 
$$|db|\le c.$$
\end{enumerate}
\end{prop}

\section{Volume growth estimate : local results}
\subsection{End of the proof of \tref{globalbis}}We are going to improve the \pref{volo} with the result of \pref{ZZlocal}. 
\begin{thm}\label{gloabB}Assume that $(M,g)$ is a complete Riemannian manifold that satisfies the Euclidean Sobolev inequality and assume that there is some compact set $K$ such that 
$$\sup_{x\in M\setminus K }\int_{M\setminus K}G(x,y)\ricm(y)\dv_{g}(y)\le \frac{1}{16n}$$ then
\begin{enumerate}[i)]
\item $(M,g)$ is doubling : for all $x\in M$ and  $R>0$: $\vol B(x,2R)\le \uptheta \vol B(x,R) .$
\item For any $x\in M$, we have 
$$H(t,x,x)\le \frac{\upgamma}{\vol B(x,\sqrt{t})}.$$
\item For $n\ge 4$ and  $p\in (1,n)$, the Riesz transform $d\Delta^{-\frac12}\colon L^{p}(M)\rightarrow L^p(T^*M)$ is bounded.
\end{enumerate}
\end{thm}
	
\begin{proof}[Proof] 
Let $o\in M$ be a fixed point and defined $r(x):=d(o,x)$ and $b(x):=G(o,x)^{-\frac{1}{n-2}}$. We know that $$|db|\le \frac{1}{\epsilon} \ \mathrm{and}\ 
\epsilon r(x)\le b(x)\le \frac{1}{\epsilon} r(x).$$

We already know that geodesics balls centered at $o$ are doubling. 
Moreover according to the lower Euclidean estimate of any geodesic balls: we have
$$ \vol B(o,r(x))\le C r(x)^n\le \vol B\left(x,r(x)/4\right)$$ 
This property is called {\it volume comparison} by A. Grigor'yan and L. Saloff-Coste and 
according to \cite[Proposition 4.7]{GS0}, in order to verify the doubling condition, it is enough to show that  the doubling property holds for remote balls : there is some $\rho>0$ such that for every $x\in M$ with $r(x)\ge \rho$ and any $r\le r(x)/4$, then $$\vol B(x,2r)\le \uptheta \vol B(x,r).$$

Choose now $\rho>0$ such that $$K\subset B\left(o,\frac{\epsilon^2}{1000}\rho\right).$$
If $x\in M$ is such that $r(x)\ge \rho$. Let $R=r(x)/2$.  One can define
$\xi_{R}=u\left(\frac{b}{R}\right)$ where $u\colon \R\rightarrow \R$ is a smooth function with compact support in $[\epsilon/4,4\epsilon]$ such that 
$u=1$ on $[\epsilon/2,2\epsilon]$. Then we have
$\xi_{R}=1$ on $B(o,2R)\setminus B(o,R/2)$ and the support of
$\xi_{R}$ is included in $B\left(o, 4\epsilon^{-2}R\right)
\setminus B(o,\frac14 \epsilon^{2}R)$.
and because
$$d\xi_{R}=\frac{1}{R}u'\left(\frac{b}{R}\right) db$$
and  \begin{align*}\Delta \xi_{R}&=\frac{1}{R}u'\left(\frac{b}{R}\right) \Delta b-\frac{1}{R^{2}}u''\left(\frac{b}{R}\right) |db|^{2}\\
&=-(n-1)\frac{1}{R}u'\left(\frac{b}{R}\right) \frac{|db|^{2}}{b}-\frac{1}{R^{2}}u''\left(\frac{b}{R}\right) |db|^{2}.
\end{align*}
Hence there is some constant  $c$ (depending only on $\epsilon$ and $u$) such that
$$|d\xi_{R}|^2+\left|\Delta \xi_R\right|\le \frac{c}{R^2}.$$

By construction, we have $\supp \xi_R\subset M\setminus K$ and $\xi_R=1$ on $B(x,r(x)/2)$.
Hence we can use the result of the \pref{ZZlocal} and get that there is a constant $\upgamma$ such that 
for all $r\in (0,r(x)/4)$:
$$\vol B(x,2r)\le \upgamma \vol B(x,r) \ \mathrm{and} \ H(r^2,x,x)\le \frac{\upgamma}{\vol B(x,r)}.$$
We have shown that the remote balls are doubling, hence $(M,g)$ is doubling. 

It remains to show the heat kernel estimate.  According to  \cite{GriRev},  the conjunction of doubling property and of the  heat kernel estimate: for all $t>0$ and all $x,y\in M$ ,
$$H_t(x,y)\le \frac{Ce^{-d^2(x,y)}{5t}}{\vol B(x,\sqrt{t})}$$
is equivalent to the so called relative Faber-Krahn inequality: 
here are positive constants $C,\mu$ such that for any $x\in M$ and $R>0$ and any open domain\footnote{We have noted $\lambda^D_1(\Omega)$ the lowest eigenvalue of the Dirichlet Laplacian on $\Omega$:
$$\lambda^D_1(\Omega)=\inf_{\varphi\in \cC^\infty_0(\Omega)}\frac{\int_\Omega |d\varphi|^2}{\int_\Omega |\varphi|^2}\,\,.$$} $\Omega\subset B(x,R)$:
$$\lambda^D_1(\Omega)\ge \, \frac{C}{ R^2}\left(\frac{\vol\Omega}{\vol B(x,R)}\right)^{-\frac2\mu}\,\, .$$
But our heat kernel estimates for remoted ball implies that the above  Faber-Krahn inequality is satisfied for remoted ball and the volume estimate and the Sobolev inequality insure that the above  Faber-Krahn inequality is satisfied for balls centered at $o$. By \cite[Proof of theorem 2.4]{CarronRMI}, the relative Faber-Krahn inequality holds on $(M,g)$.

Once these properties has been shown,the results of  \cite{DevMA} implies that when $n\ge 4$, then the Riesz transform is $L^p$ bounded for any $p\in (1,n)$.

 \end{proof}	
\subsection{Proof of the \tref{local}}\label{theoremD}
\subsubsection{The setting}
Our hypothesis and conclusion being invariant by scaling, we assume $R=1$. And we consider $(M^{n},g)$ a Riemannian manifold and $B(o,3)\subset M$ a relatively compact geodesic ball. Let $p>1$ 
$q=p/(p-1)$ and assume:
\begin{enumerate}[i)]
\item The Sobolev inequality :$\forall \psi\in \cC^{\infty}(B(o,3))$:
$$\mu \|\psi\|^{2}_{\frac{2n}{n-2}}\le \|d\psi\|_{2}^{2}.$$
\item If $G$ is the Green kernel for the Laplacian $\Delta$ with the Dirichlet boundary condition of $B(o,3)$, then
$$\sup_{x\in B(o,3)} \int_{B(o,3)}G(x,y)\ricm(y)^p\dv_{g}(y)\le \Lambda^{p}.$$
\item For some $\delta>\frac{\left(q(n-2)-2\right)^{2}}{8q(n-2)}$, the operator $\Delta-(1+\delta)(n-2)\ricm$ is non negative:
$$\forall \psi\in \cC^{\infty}(B(o,3))\colon\ (1+\delta)(n-2)\int_{B(o,3)}\ricm\psi^{2}\dv_{g}\le \int_{B(o,3)}|d\psi|^{2}\dv_{g}.$$
\end{enumerate}

We are going to prove that there is a constant $\upvartheta$ that depends only on
$n,p,\delta,\mu,\Lambda, \vol B(o,3)$ such that for all $x\in B(0,1)$ and all $r\in (0,1]$ then
$$\frac{\vol B(x,r)}{r^{n}}\le \upvartheta$$
Our objectif is to get a $L^{\infty}$ bound on the gradient of the Dirichlet Green kernel. Let $p \in B(0,1)$ and we consider
$$G(p,\bullet)=\frac{1}{b^{n-2}}$$ be the Green kernel of the Laplacian on  $B(p,1)$ for the Dirichlet boundary conditions and with pole at $p$.
We let $B:=B(p,1)$.
\subsubsection{$L^2$-estimate}
Let $\nu-2=(n-2)\sqrt{\frac{\delta}{1+\delta}}$,
the strong positivity and the universal Hardy inequality yield:
$\forall \psi\in \cC^{\infty}(B(o,1))$:
$$\frac{(\nu-2)^2}{4}\int_{B}\frac{|db|^{2}}{b^{2}}\psi^{2}\dv_{g}\le \int_{B}\left[|d\psi|^{2}-(n-2)\ricm\psi^2\right]\dv_{g},$$
If $\xi$ is a Lipschitz function with compact support in $B\setminus\{p\}$, we use the test function
$$\psi=\xi \frac{|db|^{\alpha}}{b^{n-2}}$$ for $\alpha\le n-2$.
Using the integration by part formula (\ref{lemIIP}) and the inequality (\ref{yau2})
$$\Delta\frac{|db|^{\alpha}}{b^{n-2}}-(n-2)\ricm\frac{|db|^{\alpha}}{b^{n-2}}\le 0$$ we get 
$$\frac{(\nu-2)^2}{4}\int_{B}\frac{|db|^{\alpha 2+2}}{b^{2n-2}}\xi^{2}\dv_{g}\le \int_{B}\frac{|db|^{\alpha 2}}{b^{2n-4}}|d\xi|^{2}\dv_{g}.$$
Let $\Omega\subset B$ be such that 
$$\Omega^{r}=\{x\in M, d(x,\Omega)<2\}\subset B\setminus \{p\},$$with $\xi(x):=\max\{1-d(x,\Omega)/r,0\}$ we get 
$$\frac{(\nu-2)^2}{4}\int_{\Omega}\frac{|db|^{\alpha 2+2}}{b^{2n-2}}\dv_{g}\le \frac{1}{r^{2}}\int_{\Omega^{r}}\frac{|db|^{\alpha 2}}{b^{2n-4}}\dv_{g}.$$
With H\"older inequality, we get
$$\int_{\Omega}\frac{|db|^{\alpha +1\frac{n-2}{n-1}}}{b^{2n-4}}\dv_{g}\le\left(\vol \Omega\right)^{\frac{1}{n-1}}\left(\frac{4}{(\nu-2)^2r^{2}}\right)^{\frac{n-2}{n-1}}\left( \int_{\Omega^{r}}\frac{|db|^{\alpha 2}}{b^{2n-4}}\dv_{g}\right)^{\frac{n-2}{n-1}}.$$
We are going to iterate this inequality : assume $\Omega^{r}\subset B\setminus\{p\}$ and $r=r_{1}+\dots+r_{k}$ and let $\kappa=\frac{n-2}{n-1}$
$$\alpha_{k}=(n-2)+\kappa^{k}\left(\alpha_{0}-(n-2)\right)$$ and 
$$v_{k}=\sum_{i=0}^{k-1}\kappa^{i}.$$
$$\int_{\Omega}\frac{|db|^{2\alpha_{k}}}{b^{2n-4}}\dv_{g}\le\left(\vol \Omega^{r}\right)^{\frac{v_{k}}{n-1}}\left(\frac{4}{(\nu-2)^2}\right)^{\frac{n-2}{n-1}v_{k}}\prod_{i=1}^{k}\left(\frac{1}{r_{i}^{2}}\right)^{\kappa^{i}}\left( \int_{\Omega^{r}}\frac{|db|^{2\alpha_{0}}}{b^{2n-4}}\dv_{g}\right)^{\kappa^{k}}.$$
If we choose $ r/2^{i+2}\le r_{i}\le r/2^{i}$ and if we let $k\to+\infty$, we get
that $\Omega^{r}\subset B\setminus\{p\}$, then 
\begin{equation}\label{estL2}
\int_{\Omega}\frac{|db|^{2n-4}}{b^{2n-4}}\dv_{g}\le \frac{c(n)}{(\nu-2)^{2n-4}r^{2n-4}}\vol \Omega^r.\end{equation}
\subsubsection{An integral estimate}
We introduce now the function
$$\psi:=\frac{(db|-1)^{n-2}_{+}}{b^{n-2}}.$$
We know that $\psi$ is bounded (see \ref{nearpole}) and satisfies:
$$\Delta\psi -(n-2)\ricm\psi\le (n-2)\ricm\frac{(db|-1)^{n-3}_{+}}{b^{n-2}}$$ 
Let $\tau\in (1/2,1)$ and let $\xi$ bea Lipschitz function with compact support in $B$:
We have
\begin{align*}\int_{B}|d(\xi\psi^{\tau})|^{2}\dv_{g}&=\tau\int_{B}\Delta\psi\,\psi^{2\tau-1}\xi^{2}\dv_{g}\\
&\hspace{1cm}+\int_{B}|d\xi|^{2}\psi^{2\tau}\dv_{g}+\left(\frac1\tau-1\right)\int_{B}|d\psi^{\tau}|^{2}\xi^{2}\dv_{g}.
\end{align*}

and for all $\varepsilon\in (0,1)$:
$$\int_{B}|d(\xi\psi^{\tau})|^{2}\dv_{g}\ge (1-\varepsilon)\int_{B}\xi^{2}|d\psi^{\tau}|^{2}\dv_{g}-\left(\frac1\varepsilon -1\right)\int_{B}|d\xi|^{2}\psi^{2\tau}\dv_{g}$$
so that
$$\int_{B}\xi^{2}|d\psi^{\tau}|^{2}\dv_{g}\le \frac{1}{1-\varepsilon} \int_{B}|d(\xi\psi^{\tau})|^{2}\dv_{g}+\frac1\varepsilon\int_{B}|d\xi|^{2}\psi^{2\tau}\dv_{g}.$$
And we get
\begin{equation}\begin{split}
\left(1-\frac{1-\tau}{\tau(1-\varepsilon)}\right)\int_{B}|d(\xi\psi^{\tau})|^{2}\dv_{g}&\le \tau\int_{B}\Delta\psi\,\psi^{2\tau-1}\xi^{2}\dv_{g}\\
&+\left(1+\frac1\varepsilon\left(\frac1\tau-1\right)\right)\int_{B}|d\xi|^{2}\psi^{2\tau}\dv_{g}
\end{split}
\end{equation}
According to our hypothesis, on $\delta$, we can choose $\tau\in (1/2,1)$, $\varepsilon\in (0,2-1/\tau)$ such that
\begin{equation}\label{consttau}
2\tau-1<\frac{2}{q(n-2)}
\end{equation}
$$\kappa:=\frac{2\tau-1-\varepsilon\tau}{\tau(1-\varepsilon)}-\frac{\tau}{1+\delta}>0.$$
Let $c:=\left(1+\frac1\varepsilon\left(\frac1\tau-1\right)\right)$, 
we get:
\begin{equation*}\begin{split}
\kappa \int_{B}|d(\xi\psi^{\tau})|^{2}\dv_{g}&\le \left(1-\frac{1-\tau}{\tau(1-\varepsilon)}\right)\int_{B}|d(\xi\psi^{\tau})|^{2}\dv_{g}-\tau(n-2)\int_{B}\ricm\psi^{2\tau}\xi^{2}\dv_{g}\\
&\le \tau\int_{B}\left[\Delta\psi-(n-2)\ricm\psi\right]\,\psi^{2\tau-1}\xi^{2}\dv_{g}
+c\int_{B}|d\xi|^{2}\psi^{2\tau}\dv_{g}\\
&\le \tau(n-2)\int_{B}\ricm \frac{(db|-1)^{2\tau(n-2)-1}_{+}}{b^{2\tau(n-2)}}\xi^{2}\dv_{g}
+c\int_{B}|d\xi|^{2}\psi^{2\tau}\dv_{g}.
\end{split}
\end{equation*}
We choose now
$$\xi(x)=\max\{1-2d(p,x),\frac12\}.$$
Using our $L^{2}$ estimate (\ref{estL2}), we get
$$c\int_{B}|d\xi|^{2}\psi^{2\tau}\dv_{g}\le 4c\int_{B(p,1/2)\setminus B(p,1/4)}\psi^{2\tau}\dv_{g} \le C \vol B(p,1).$$
Let $Q=2\tau(n-2)$, the H\"older inequality yields
\begin{equation*}\begin{split}
\int_{B}\ricm \frac{(db|-1)^{2\tau(n-2)-1}_{+}}{b^{2\tau(n-2)}}\xi^{2}\dv_{g}&\le \left(\int_{B}\ricm \psi^{2\tau}\xi^{2}\dv_{g}\right)^{1-\frac{1}{Q}}\left(\int_{B}\ricm \frac{\xi^{2}}{b^{Q}}\dv_{g}\right)^{\frac{1}{Q}}\\
&\le\lambda \int_{B}\ricm \psi^{2\tau}\xi^{2}\dv_{g}+\lambda^{1-Q}\int_{B}\ricm \frac{\xi^{2}}{b^{Q}}\dv_{g}\\
&\le \lambda \int_{B}|d(\xi\psi^{\tau})|^{2}\dv_{g}+\lambda^{1-Q}\int_{B}\ricm \frac{\xi^{2}}{b^{Q}}\dv_{g}
\end{split}
\end{equation*}
We use now $\lambda$ such that $\tau(n-2)\lambda=\kappa/2$ and we get

$$ \frac{\kappa}{2}\int_{B}|d(\xi\psi^{\tau})|^{2}\dv_{g}\le C \int_{B}\ricm \frac{\xi^{2}}{b^{Q}}\dv_{g}+C \vol B(p,1).$$
But  using again the H\"older inequality, we have:
\begin{equation*}\begin{split}
\int_{B}\ricm \frac{\xi^{2}}{b^{2\tau(n-2)}}\dv_{g}&\le \left(\int_{B}\ricm^{p} \frac{1}{b^{n-2}}\dv_{g}\right)^{\frac1p}\left(\int_{B}\frac{1}{b^{((2\tau-1)q+1)(n-2)}}\dv_{g}\right)^{\frac1q}\\
&\le \Lambda C\left(\frac1\mu\right)^{2\tau-1+\frac1q}\left(\vol B\right)^{1-2\tau+\frac2n\left(2\tau-1+\frac1q\right)}.\end{split}
\end{equation*}
With 
$$\mathbf{I}:=\Lambda\left(\frac{\vol B(o,3)}{\mu^{\frac n2}}\right)^{\frac{2}{nq}}$$ we get
$$\int_{B}\ricm \frac{\xi^{2}}{b^{2\tau(n-2)}}\dv_{g}\le c \mathbf{I} \left(\frac{\left(\vol B(o,3)\right)^{\frac 2n}}{\mu}\right)^{2\tau-1}\left(\vol B(o,3)\right)^{1-2\tau}.$$
Recall that according to \tref{sobolev}-iv, we get 
$$\vol B(o,3)\ge c_n\mu^{\frac n2}$$ hence
$$\int_{B}\ricm \frac{\xi^{2}}{b^{2\tau(n-2)}}\dv_{g}\le c \mathbf{I} \left(\frac{\left(\vol B(o,3)\right)^{\frac 2n}}{\mu}\right)^{2\tau-1}  {\mu}^{n(\tau-1/2)}.$$
Using the universal Hardy inequality:
$$\frac{(n-2)^{2}}{4}\int_{B}\frac{|db|^{2}}{b^{2}}(\xi\psi^{\tau})^{2}\dv_{g}\le \int_{B}|d(\xi\psi^{\tau})|^{2},$$ one gets:
\begin{equation}\label{intestimedb}
\int_{B(p,1/4)}\frac{|db|^{2} (db|-1)^{2\tau(n-2)}_{+}}{b^{2\tau(n-2)+2}}\dv_{g}\le \Gamma
\end{equation}
with 
$$\Gamma=c\left( \vol B(o,3)+\mathbf{I} \left(\left(\vol B(o,3)\right)^{\frac 2n}/\mu\right)^{2\tau-1} {\mu}^{n(\tau-1/2)}\right).$$
\subsubsection{Bound on the gradient}
Recall that 
$$\frac{1}{b^{n-2}(x)}\le \frac{c_{n}}{\mu^{\frac n2}d(p,x)^{n-2}}$$
Hence if
$R\le \epsilon_{n}\mu^{\frac{n}{2(n-2)}}$ then
$$\Omega_{R}^{\#}=\left\{\frac R2 \le b\le \frac52 R\right\}\subset B(p,1/4).$$
Using 
$x^{2\tau(n-2)}\le 2^{2\tau(n-2)}\left( 1+(x-1)_{+}^{2\tau(n-2)}\right)$ and\footnote{this is a consequence of the Green formula and of the coaera formula:
$$\int_{\Omega_{R}^{\#}}|db|^{2}\dv_{g}=\int_{R/2}^{5R/2}\left( \int_{b=t}|db|\right) dt$$ and by the Green formula $ \int_{b=t}\frac{|db|}{b^{n-1}}=c_{n}$}
$$\int_{\Omega_{R}^{\#}}|db|^{2}\dv_{g}=c_{n}R^{n}$$
we get 
$$\int_{\Omega_{R}^{\#}}|db|^{2+2\tau(n-2)}\dv_{g}\le cR^{n}+c\Gamma R^{2+2\tau(n-2)}.$$
Our hypothesis provide a constant $\gamma$ depending only on 
$\delta,n,p, I$ and $\left(\vol B(o,3)\right)^{\frac 2n}/\mu$ such that there is a solution
$\Delta h=\frac{n-2}{n-1}\ricm h$ on $B(p,1)$ with $1\le h\le \gamma$.  We let 
$$\rho:= \epsilon_{n}\mu^{\frac{n}{2(n-2)}}$$  and with \pref{estimedb}, we get that
if $b\le \rho$ then
$$|db|^{(2\tau-1)(n-2)}\le B^{(2\tau-1)(n-2)}$$
where
$$B^{(2\tau-1)(n-2)}:=c\left(1+\Gamma \rho^{(2\tau-1)(n-2)}\right)\frac{\gamma^{2+2\tau(n-2)\frac{n-1}{n-2}+n-2}}{\mu^{\frac n2}}.$$
\subsubsection{Volume upper bound}
If $d(p,x)\le \rho/B$ then we have
$$b(x)\le Bd(p,x)$$ and hence
for $r\le \rho$, we get
$$\frac{\vol B(p,r)}{r^{n}}\le \frac{\vol \{b\le Br\}}{r^{n}}\le B^{n}\mu^{-\frac{n}{n-2}}.$$
Where as for $\rho\le r\le 1$, one gets:
$$\frac{\vol B(p,r)}{r^{n}}\le\frac{\vol B(o,3)}{\rho^{n}}.$$
\subsubsection{Further consequence} It remains to show how one can get the Poincar\'e inequality, it is in fact a direct consequence of the following proposition that could have been used in order to prove the \tref{globalsuite}.
\begin{prop}\label{poincarfromsob} Assume that $B(x,2R)$ is a relatively compact geodesic ball in a Riemannian manifold and assume that $B(x,2R)$ satisfies the Euclidean Sobolev inequality:$\forall \psi\in \cC^{\infty}(B(o,2R))$:
$$\mu \|\psi\|^{2}_{\frac{2n}{n-2}}\le \|d\psi\|_{2}^{2}.$$
and assume that the \sch operator $\Delta-\ricm$ is jaugeable: there is $h\colon B(x,2R)\rightarrow \R$ such that
$$\Delta h-\ricm h=0\ \mathrm{and}\ 1\le h\le \gamma$$
Then for 
$$\uplambda= c_n\gamma^n\frac{\vol B(x,2R)}{\mu^{\frac n2}R^n}$$ we have the Poincar\'e type inequality :
$$\forall \psi\in \cC^1(B(x,2R))\colon\ \int_{B(x,R)} (\psi-\psi_{B(x,R)})^2\dv_g\le \uplambda R^2 \int_{B(x,2R)} |d\psi|_g^2\dv_g.$$
\end{prop}
\proof
Let $\psi\in \cC^1(B(x,2R))$ and let $\varphi$ be the harmonic extension of $\left.\psi\right|_{\partial B(x,2R)}.$
The Sobolev inequality implies that
\begin{equation}\label{Diri} \|\psi-\varphi\|_2^2\le \frac{\left(\vol B(x,2R)\right)^{\frac 2n}}{\mu} \|d\psi-d\varphi\|_2^2\end{equation}
The function $|d\varphi|$ satisfies 
$$\Delta d\varphi|\le \ricm |d\varphi|,$$ hence one get from \tref{sobolevgamma}-iii)
$$\sup_{z\in B(x,R)} |d\varphi|^2(z)\le \frac{c_n \gamma^n}{\mu^{n/2}R^n}\int_{B(x,2R)}|d\varphi|^2\dv_g.$$
In particular with $c=\varphi(x)$, one gets:
\begin{align}\label{harmo} \|\varphi-c\|_2^2&\le R^2\vol B(x,R)\sup_{z\in B(x,R)} |d\varphi|^2(z) \notag\\
&\le  c_n R^2\gamma^n \frac{\vol B(x,2R)}{\mu^{n/2}R^n}\int_{B(x,2R)}|d\varphi|^2\dv_g.\end{align}
The conclusion follows now from \ref{Diri} and \ref{harmo} and the fact that the ratio 
$$\vol B(x,2R)/(\mu^{n/2}R^n)$$ is bounded from below by a constant that depends only on $n$.\endproof
In the setting of this subsection (\ref{theoremD}), we have proven that all there is a positive constant $\uptheta$ such that forall $x\in B(o,1)$ and any $r\in (0,1)$ then
$$\frac{r^n}{\uptheta}\le \vol B(x,r)\le \uptheta r^n.$$
Note that by monotonicity of $r\mapsto \vol B(x,r)$ the same kind of inequality is true for all $r\in (0,2).$
So that by the above \pref{poincarfromsob}, we get that there is a constant $\uplambda$ such that for any $x\in B(o,1)$ and any $r\in (0,1)$:

$$\forall \psi\in \cC^1(B(x,2r))\colon\ \int_{B(x,r)} (\psi-\psi_{B(x,r)})^2\dv_g\le \uplambda r^2 \int_{B(x,2r)} |d\psi|_g^2\dv_g.$$
The anounced Poincar\'e inequalities then follow from a now classical results of D. Jerison (\cite{jerison, MSC}.

\end{document}